\documentclass{article}
\usepackage{amsmath,amssymb,amsthm,amscd,latexsym,}
\usepackage{graphicx}

\makeatletter
\renewcommand{\mod}[1]{\allowbreak \if@display \mkern 8mu \else
\mkern 5mu\fi {\operator@font mod}\,\,#1}
\makeatother

\newcommand{\bn}{\mathbb N}
 \newcommand{\bq}{\mathbb Q}
\newcommand{\br}{\mathbb R}
\newcommand{\bz}{\mathbb Z}
\newcommand{\bff}{\mathbb F}

\newcommand{\bk}{\mathbb K}

\newtheorem{theorem}{Theorem}[section]

\newtheorem{definition}[theorem]{Definition}

\newtheorem{lemma}[theorem]{Lemma}

\newcommand\F{\mathcal F}
\newcommand\Hh{\mathcal H}

\newcommand\M{\mathcal M}

\begin{document}
\title{On ground fields of arithmetic hyperbolic reflection groups. II}
\date{}
\author{Viacheslav V. Nikulin\footnote{This paper was written with the
financial support of EPSRC, United Kingdom (grant no. EP/D061997/1)}}
\maketitle

\begin{abstract}
This paper continues \cite{Nik5} (arXiv.org:math.AG/0609256) and
\cite{Nik6} \linebreak
(arXiv:0708.3991).

Using authors's methods of 1980, 1981,
some explicit finite sets of number fields containing all ground fields of
arithmetic hyperbolic reflection groups in dimensions at least 4
are defined, and good explicit bounds of
their degrees (over $\bq$) are obtained. This could be important for
further classification.

Thus, now, an explicit bound of degree of ground fields of arithmetic
hyperbolic reflection groups is
unknown in dimension 3 only.
\end{abstract}

\vskip1cm \centerline{\it To 70th Birthday of \'Ernest Borisovich Vinberg}
\vskip1cm

\section{Introduction} \label{introduction} This paper continues \cite{Nik6}.
See the introduction of this paper about
history, definitions and results concerning the subject.

In \cite{Nik6} some explicit finite sets of totally real algebraic number
fields containing all ground fields of arithmetic hyperbolic reflection
groups in dimensions $n\ge 6$ were defined, and good explicit bounds of
degrees (over $\bq$) of their fields were obtained.  In particular,
an explicit bound ($\le 56$) of degree of the ground field of any
arithmetic hyperbolic reflection group in
dimension $n \ge 6$ was obtained.

Here we continue this study for smaller dimensions $n=5$ and $4$.
Using similar methods, we define some
explicit finite sets of totally real algebraic number fields
containing all ground fields of arithmetic hyperbolic reflection groups
in dimensions $n \ge 4$. Moreover, an explicit bound ($\le 138$)
of degrees of fields from these sets are obtained. This
requires much more difficult considerations comparing to $n\ge 6$,
and it is very surprising to the author that this can be done.
Thus, degree of the ground field of any arithmetic hyperbolic reflection
group of dimension $n\ge 4$ is bounded by $138$.

It is also very important that all these fields are attached to fundamental
chambers of arithmetic hyperbolic reflection groups, and they can be further
geometrically investigated and restricted.

It was also shown in \cite{Nik6} (using results of \cite{LMR} and
\cite{Borel}, \cite{Tak4}) that degree of the ground field of
any arithmetic hyperbolic reflection group of dimension $n=2$ is
bounded by $44$. Thus, an explicit bound of degree
of ground fields of arithmetic hyperbolic reflection groups remains
unknown in dimension $n=3$ only. Finiteness of the number of
maximal arithmetic hyperbolic reflection groups (and then a theoretical
existence of this bound) was shown by Agol \cite{Agol}.

Since this paper is a direct continuation of \cite{Nik6}, we use
notations, definitions and results of this paper without their reminding.

In \cite{Vin2} (1981) and \cite{Vin3}, \'Ernest Borisovich Vinberg
gave a very short list of all possible ground fields
(thirteen fields) of arithmetic
hyperbolic reflection groups in dimensions $n\ge 14$. Thus, the results
of \cite{Nik6} and of this paper can be viewered as some extension of these
beautiful Vinberg's results to smaller dimensions.

%%%%%%%%%%%%%%%%%%%%%%%%%%%%%%%%%%%%%%%%%
%%%%%%%%%%%%%%%%%%%%%%%%%%%%%%%%%%%%%%%%%
%%%%%%%%%%%%%%%%%%%%%%%%%%%%%%%%%%%%%%

\section{Ground fields of arithmetic hyperbolic reflection
groups in dimensions $n\ge 4$}
\label{sec:5and4}

Since this paper is a direct continuation of \cite{Nik6}, we use
notations, definitions and results of this paper without their reminding.

In \cite[Secs 3 and 4]{Nik6}, explicit finite sets
$\F L^4$, $\F T$, $\F\Gamma^{(4)}_i(14)$, $1\le i\le 5$,
and $\F\Gamma_{2,4}(14)$ of totally real algebraic number fields
were defined. The set $\F L^4$ consists of all ground fields of arithmetic
Lann\'er diagrams with $\ge 4$ vertices and consists of three fields of
degree $\le 2$. The set $\F T$ consists of all ground fields of arithmetic
triangles (plane) and has 13 fields of degree $\le 5$ (it includes $\F L^4$).
The set $\F\Gamma^{(4)}_i(14)$, $1\le i\le 5$, consists of all
ground fields of V-arithmetic edge polyhedra of minimality $14$ with
connected Gram graph having 4 vertices. They are determined by $5$ types
of graphs $\Gamma^{(4)}_i(14)$, $i=1,2,3,4,5$. The degrees of fields from
these sets are bounded by 22, 39, 53, 56, 54 respectively.
The set $\F\Gamma_{2,4}(14)$ consists of all ground fields of arithmetic
quadrangles (plane) of minimality $14$. Their degrees are bounded by 11.

%%%%%%%%%%%%%%%%%%%%%%%%%%%%%%%
%%%%%%%%%%%%%%%%%%%%%%%%%%%%%%
%%%%%%%%%%%%%%%%%%%%%%%%%%%%%%%

The following result was obtained in \cite[Theorem 4.5]{Nik6}
using methods of \cite{Nik1} and \cite{Nik2} and results by Borel
\cite{Borel} and Takeuchi \cite{Tak4}.

\begin{theorem} (\cite{Nik6})
In dimensions $n\ge 6$, the ground field of
any arithmetic hyperbolic reflection group belongs to one of finite
sets of fields $\F L^4$, $\F T$, $\F\Gamma^{(4)}_i(14)$, $1\le i\le 5$,
and $\F\Gamma_{2,4}(14)$. In particular, its degree is bounded by $56$.
\label{thfor6}
\end{theorem}

Applying the same methods and similar, but much more difficult
considerations, here we want to extend this result to $n\ge 4$,
also considering $n=5$ and $n=4$.

\medskip

First, we introduce some other explicit finite sets of fields.
All of them are related to fundamental pentagons on
hyperbolic plane.  Similarly $\F T$ is related to arithmetic
triangles on hyperbolic plane, and $\F\Gamma_{2,4}(14)$ is related to
arithmetic quadrangles on hyperbolic plane.

Let us consider plane
(or Fuchsian) arithmetic hyperbolic reflection groups $W$ with a
pentagon fundamental polygon
$\Delta$ of minimality $14$. We remind that this means that the set 
$P(\Delta)=\{\delta_1,\delta_2,\delta_3,\delta_4,\delta_5\}$ of all 
perpendicular with square $(-2)$ and directed outwards vectors to 
codimension one faces of $\Delta$ satisfies the condition
$$
\delta_i\cdot \delta_j<14,\ \forall\ \delta_i,\delta_j\in P(\Delta).
$$
Respectively, we call $\Delta$ as an
{\it arithmetic pentagon of minimality $14$.}

\begin{definition} We denote by $\Gamma_{2,5}(14)$ the set of
Gram graphs $\Gamma (P(\Delta))$ of all arithmetic pentagons $\Delta$
of minimality $14$.
The set $\F\Gamma_{2,5}(14)$ consists of all their ground fields.
\label{F{2,5}(14)}
\end{definition}

By Borel \cite{Borel} and Takeuchi \cite{Tak4}, for fixed $g\ge 0$
and $t\ge 0$,
the number of arithmetic Fuchsian groups with signatures
$(g;\ e_1,\,e_2,\,\dots ,\, e_t)$
is finite. Applying this result to $g=0$ and $t=5$, we obtain that sets of
arithmetic quadrangles $\Gamma_{2,5}$ and their ground fields $\F \Gamma_{2,5}$
are finite. Then their subsets $\Gamma_{2,5}(14)$ and
$\F\Gamma_{2,5}(14)$ are also finite.

Moreover, in \cite[pages 383--384]{Tak4} an upper bound $n_0$ of the degree
of ground fields of Fuchsian groups with signatures
$(g;\ e_1,\,e_2,\,\dots,\, e_t)$ is given. It is
\begin{equation}
n_0=(b+\log_e{C(g,t)})/\log_e(a/(2\pi)^{4/3})
\label{Tak1}
\end{equation}
where
$$
a=29.099,\ b=8.3185,\ C(g,t)=2^{2g+t-2}(2g+t-2)^{2/3}
$$
(here $a$ and $b$ are due to Odlyzko).
It follows that
\begin{equation}
[\bk:\bq]\le 12\ \ {\rm for\ \ } \bk\in \F\Gamma_{2,5}\supset 
\F\Gamma _{2,5}(14)
\label{degF2,4(14)}.
\end{equation}

\begin{figure}
\begin{center}
\includegraphics[width=12cm]{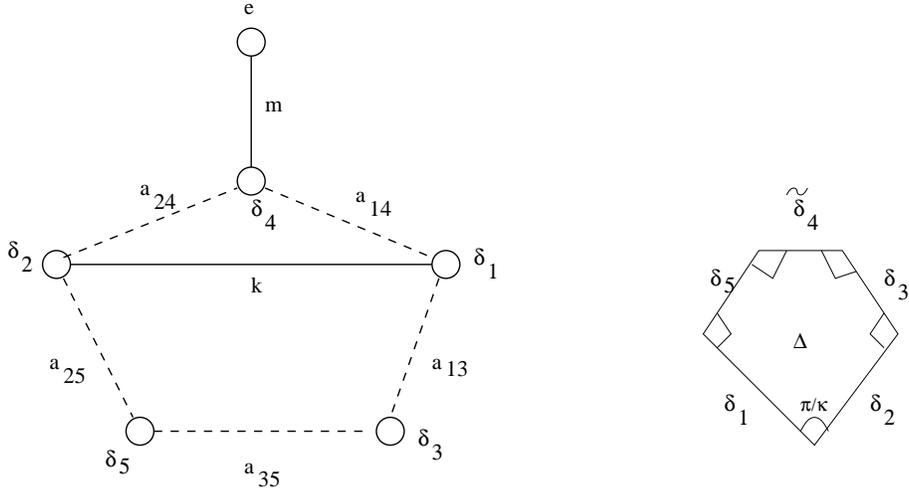}
\end{center}
\caption{The pentagon graphs $\Gamma_1^{(6)}$}
\label{pentgraph61}
\end{figure}

Let us consider V-arithmetic 3-dimensional chambers
which are defined by the Gram graphs $\Gamma^{(6)}_1$ with 6
vertices $\delta_1,\dots,\delta_5, e$
shown in Figure \ref{pentgraph61}.
It follows that the corresponding V-arithmetic chamber $\M$ satisfies
the following condition: the 2-dimensional face $\M_e$ of $\M$ which is
perpendicular to $e$ is a pentagon $\Delta$ where
$$
P(\Delta)=\{\delta_1,\delta_2,\delta_3, \widetilde{\delta}_4,
\delta_5\}
$$
for
$$
\widetilde{\delta}_4=\frac{\delta_4+\cos{(\pi/m)}e}{\sin{(\pi/m)}},\
$$
are perpendicular to 5 consecutive sides of the pentagon.
It has one angle $\pi/k$ (defined by $\delta_1,\delta_2$) and all other its
angles are right ($=\pi/2$). Thus, all planes $\Hh_{\delta_i}$, $1\le i\le 5$,
are perpendicular to $\Hh_e$ except $\Hh_{\delta_4}$ which has angle $\pi/m$,
$m\ge 3$, with the plane $\Hh_e$.

\begin{definition} We denote by $\Gamma^{(6)}_1(14)$ the set of
all V-arithmetic 3-dimensional graphs $\Gamma^{(6)}_1$ (or the corresponding
3-dimensional V-arithmetic chambers) of minimality $14$.
Thus, inequalities $2<a_{ij}<14$ satisfy.
We denote by $\F \Gamma^{(6)}_1(14)$ the set of all their ground fields.
\label{defpentgamma61fields}
\end{definition}

Similarly, we define V-arithmetic graphs $\Gamma^{(6)}_2$,
$\Gamma^{(6)}_3$ given in Figure \ref{pentgraphs}
which must (by definition) define 3-dimensional
V-arithmetic chambers $\M$. For both of them, the face $\M_e$ of $\M$
which is perpendicular to $e$,
is a pentagon $\Delta$ with right angles. For $\Gamma^{(6)}_2$,
$$
P(\Delta)=\{\widetilde{\delta}_1, \delta_2,\delta_3, \delta_4, \delta_5\}
$$
correspond to five consecutive sides of $\Delta$ where
$$
\widetilde{\delta}_1=\frac{\delta_1+\cos{(\pi/m)}e}{\sin{(\pi/m)}}.\
$$
For $\Gamma^{(6)}_3$,
$$
P(\Delta)=\{\widetilde{\delta}_1, \delta_2,\widetilde{\delta_3},
\delta_4, \delta_5\}
$$
correspond to five consecutive sides of $\Delta$ where
$$
\widetilde{\delta}_1=\frac{\delta_1+\cos{(\pi/m_1)}e}{\sin{(\pi/m_1)}},\ \ \
\widetilde{\delta}_3=\frac{\delta_3+\cos{(\pi/m_3)}e}{\sin{(\pi/m_3)}}.
$$
V-arithmetic graphs $\Gamma^{(7)}_1$ and $\Gamma^{(7)}_2$ of Figure
\ref{pentgraphs} must (by definition) define 4-dimensional V-arithmetic
fundamental chambers $\M$. For both of them, the 2-dimensional
face $\M_{e_1,e_2}$ of $\M$ which is perpendicular to both $e_1$ and $e_2$,
is a pentagon $\Delta$ with right angles. For $\Gamma^{(7)}_1$,
$$
P(\Delta)=\{\widetilde{\delta}_1, \widetilde{\delta_2},
\delta_3, \delta_4, \delta_5\}
$$
correspond to five consecutive sides of $\Delta$ where
$$
\widetilde{\delta}_1=\frac{\delta_1+\cos{(\pi/m_1)}e_1}{\sin{(\pi/m_1)}},\ \ \
\widetilde{\delta}_2=\frac{\delta_2+\cos{(\pi/m_2)}e_2}{\sin{(\pi/m_2)}}.
$$
For $\Gamma^{(7)}_2$,
$$
P(\Delta)=\{\widetilde{\delta}_1, \widetilde{\delta}_2,
\widetilde{\delta}_3, \delta_4, \delta_5\}
$$
correspond to five consecutive sides of $\Delta$ where
$$
\widetilde{\delta}_1=\frac{\delta_1+\cos{(\pi/m_1)}e_1}{\sin{(\pi/m_1)}},\ \ \
\widetilde{\delta}_2=\frac{\delta_2+\cos{(\pi/m_2)}e_2}{\sin{(\pi/m_2)}},
\ \ \
\widetilde{\delta}_3=\frac{\delta_3+\cos{(\pi/m_3)}e_1}{\sin{(\pi/m_3)}}.\ \ \
$$

\begin{figure}
\begin{center}
\includegraphics[width=12cm]{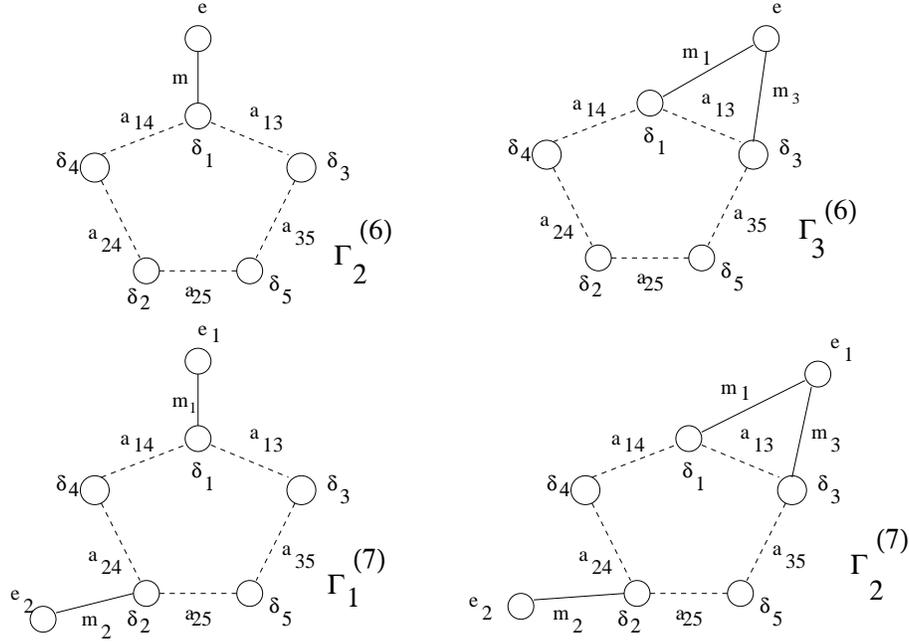}
\end{center}
\caption{The pentagon graphs $\Gamma_2^{(6)}$, $\Gamma_3^{(6)}$,
$\Gamma_1^{(7)}$, $\Gamma_2^{(7)}$}
\label{pentgraphs}
\end{figure}

\begin{definition} We denote by $\Gamma^{(6)}_2(14)$,
$\Gamma^{(6)}_3(14)$, $\Gamma^{(7)}_1(14)$ and $\Gamma^{(7)}_2(14)$
the sets of all V-arithmetic graphs
$\Gamma^{(6)}_2$, $\Gamma^{(6)}_3$, $\Gamma^{(7)}_1$ and
$\Gamma^{(7)}_2$ respectively (or the corresponding
3-dimensional or 4-dimensional V-arithmetic fundamental chambers)
of minimality $14$.
Thus, inequalities $2<a_{ij}<14$ satisfy.
We denote by
$\F\Gamma^{(6)}_2(14)$,
$\F\Gamma^{(6)}_3(14)$, $\F\Gamma^{(7)}_1(14)$ and
$\F\Gamma^{(7)}_2(14)$ the sets of all their ground fields respectively.
\label{defpentagonfields}
\end{definition}

We have

\begin{theorem} The sets of V-arithmetic graphs
$\Gamma^{(6)}_1(14)$,
$\Gamma^{(6)}_2(14)$, $\Gamma^{(6)}_3(14)$, $\Gamma^{(7)}_1(14)$,
$\Gamma^{(7)}_2(14)$ and their fields $\F\Gamma^{(6)}_1(14)$,
$\F\Gamma^{(6)}_2(14)$,
$\F\Gamma^{(6)}_3(14)$, $\F\Gamma^{(7)}_1(14)$,
$\F\Gamma^{(7)}_2(14)$ are finite.

Degree of any field from $\F\Gamma^{(6)}_1(14)$ is bounded ($\le$) by $56$.

Degree of any field from $\F\Gamma^{(6)}_2(14)$ is bounded by $75$.

Degree of any field from  $\F\Gamma^{(6)}_3(14)$ is bounded by $138$.

Degree of any field from $\F\Gamma^{(7)}_1(14)$ is bounded by $42$.

Degree of any field from  $\F\Gamma^{(7)}_2(14)$ is bounded by $138$.
\label{thpentgraphs}
\end{theorem}

\begin{proof} The proof requires long considerations and calculations. 
It will be given in a special Sect. \ref{sec:pentgraphs}.
\end{proof}

We have the following main result of the paper.

\begin{theorem} In dimensions $n\ge 4$, the ground field of
any arithmetic hyperbolic reflection group belongs to one of finite
sets of fields $\F L^4$, $\F T$, $\F\Gamma^{(4)}_i(14)$, $1\le i\le 5$,
$\F\Gamma_{2,4}(14)$ and
$\F\Gamma^{(6)}_1(14)$,
$\F\Gamma^{(6)}_2(14)$,
$\F\Gamma^{(6)}_3(14)$, $\F\Gamma^{(7)}_1(14)$,\linebreak
$\F\Gamma^{(7)}_2(14)$,
$\F \Gamma_{2,5}(14)$.

In particular, its degree is bounded by $138$.
\label{thfor45}
\end{theorem}

\begin{proof} By \cite{Nik6}, if $n\ge 6$, the ground field $\bk$
belongs to
one of sets $\F L^4$, $\F T$, $\F\Gamma^{(4)}_i(14)$, $1\le i\le 5$,
$\F\Gamma_{2,4}(14)$. Thus, further we can assume that the
ground field $\bk$ does not belong to these sets,
and the dimension is equal to $n=4$ or $n=5$.

Let $W$ be an arithmetic hyperbolic reflection group of dimension $n=4$ or
$5$, $\M$ is its fundamental chamber, and $P(\M)$ is the set of all
vectors with square $(-2)$ which are perpendicular to codimension one faces
of $\M$ and directed outwards. For $\delta \in P(\M)$ we denote by
$\Hh_\delta$ and  $\M_\delta$ the hyperplane and the codimension one face
$\M\cap \Hh_e$ respectively which is perpendicular to $e$.

By \cite{Nik1}, there exists $e\in P(\M)$ which defines
a narrow  face $\M_e$ of $\M$ of minimality $14$.
It means that $\delta_1\cdot \delta_2<14$ for any
$\delta_1,\delta_2\in P(\M,e)\subset P(\M)$. Here
$$
P(\M,e)=\{\delta \in P(\M)\ |\ \Hh_\delta\cap \Hh_e\not=\emptyset\}.
$$

The field $\bk$ is different from $\bq$. It is well-known that then $\M$
is compact, and it is simple: any vertex of $\M$ is contained in exactly
$n$ codimension one faces. Then $\M_e$ is also simple $n-1$-dimensional
polyhedron. If $n=5$, it is 4-dimensional, and if $n=4$, it is 3-dimensional.

By formula \cite[(7)]{Nik6} which was proved in \cite{Nik2}, the average
number $\alpha_{n-1}^{(0,2)}$ of vertices of 2-dimensional faces of any
simple $n-1$-dimensional polyhedron satisfies for $n\ge 4$ the inequality
\begin{equation}
\alpha_{n-1}^{(0,2)}<4+
\left\{
\begin{array}{rcl}
\frac{4}{n-2} & {\rm if } &\  n\ {\rm is\ even,}\\
\frac{4}{n-3} & {\rm if}   &\  n\ {\rm is\ odd.}
\end{array}
\right.
\label{alpha02}
\end{equation}
It follows that $\alpha_{n-1}^{(0,2)}<6$, and $\M_e$ has a 2-dimensional
face which is a polygon $\Delta$ with $\le 5$ vertices.
Remark that the same can be obtained by Euler's formula for Euler
characteristic $b_0-b_1+b_2=2$ applied
to any 3-dimensional face of $\M_e$. Thus, for
$n=4$ and $n=5$, actually, we don't need the inequality \eqref{alpha02}.

If $\Delta$ is a triangle, then obviously $\bk$ belongs to $\F L^4$ or $\F T$
(see \cite[Lemma 4.6]{Nik6}).
If $\Delta$ is a quadrangle and $\bk$ is different from $\F L^4$,
$\F T$ and $\F \Gamma_i^{(4)}(14)$, $1\le i\le 5$, then $\bk$ belongs to
$\F\Gamma_{2,4}(14)$, see \cite[Lemma 4.7]{Nik6}. Thus, further, we can
assume that $\Delta$ is a pentagon.

By \cite[Lemma 4.3]{Nik6}, if $\bk$ does not belong to
$\F L^4$, $\F T$ and $\F \Gamma_i^{(4)}(14)$, $1\le i\le 4$,
then the Coxeter graph $C(v)$ of any vertex $v\in \M_e$
has all connected components having only one or
two vertices. If additionally $\bk$ does not belong to
$\F \Gamma_5^{(5)}(14)$, then the hyperbolic connected component of
the edge graph $\Gamma(r)$ defined by any edge
$r=v_1v_2\subset \M_e$ has $\le 3$ vertices.

By direct check (see the proof of \cite[Lemma 4.7]{Nik6}), we see that then
there are only 3 cases where we denote by $Q\subset P(\M)$ the set of all
perpendicular to $\Delta$ elements $\delta \in P(\M)$, and by
$\delta_i\in P(\M)-Q$, $i=1,2,3,4,5$, perpendicular vectors to five
consecutive edges of $\Delta$. Thus, $Q$ consists of $n-2$ elements, and the
plane of $\Delta$ is intersection of hyperplanes
$\Hh_\delta$, $\delta\in Q$. Moreover, $\Hh_{\delta_i}\cap \Delta$,
$i=1,2,3,4,5$,  give lines of five consecutive edges of $\Delta$.

{\bf Case 1.} $Q\perp \delta_i$, $1\le i\le 5$. Then $\Delta$ is a fundamental
pentagon of minimality $14$ for arithmetic hyperbolic reflection group with
$P(\Delta)=\{\delta_1,\delta_2,\delta_3, \delta_4,\delta_5\}$. Then $\bk$
is the ground field of this group and $\bk\in \F\Gamma_{2,5}(14)$.

{\bf Case 2.} $Q$ is not perpendicular to all $\delta_i$, $1\le i\le 5$, and
$\Delta$ has a non-right angle. We can assume that the vertex of
this angle is perpendicular to $\delta_1$ and $\delta_2$. Thus,
$\delta_1\cdot \delta_2=2\cos(\pi/k)$, $k\ge 3$. Then all other angles of
$\Delta$ are $\pi/2$. There exists exactly one element $e\in Q$ which
is not perpendicular all $\delta_i$, $1\le i\le 5$. Then $e$ is perpendicular
to all $\delta_i$ except $\delta_4$, and $e\cdot \delta_4=2\cos(\pi/m)$,
$m\ge 3$. Thus, $e$ and $\delta_i$, $1\le i\le 5$, define the hyperbolic
graph $\Gamma_1^{(6)}(14)$, and $\bk$ is the ground field of this
hyperbolic graph. (This case is very similar to \cite[Lemma 4.7]{Nik6}.)

{\bf Case 3.}
$Q$ is not perpendicular to all $\delta_i$, $1\le i\le 5$, but all angles of
$\Delta$ are right. Then $\delta_i$, $1\le i\le 5$, and elements from $Q$
which are not perpendicular to all these elements define one of
hyperbolic graphs
$\Gamma^{(6)}_2(14)$, $\Gamma^{(6)}_3(14)$, $\Gamma^{(7)}_1(14)$,
$\Gamma^{(7)}_2(14)$. The field $\bk$ is the ground field of one of these
hyperbolic graphs.

We leave the corresponding routine and simple check to a reader.

This finishes the proof of the theorem. \end{proof}

%%%%%%%%%%%%%%%%%%%%%%%%%%%%%%%%%%%%%%%%%%%
%%%%%%%%%%%%%%%%%%%%%%%%%%%%%%%%%%%%%%%%%%%
%%%%%%%%%%%%%%%%%%%%%%%%%%%%%%%%%%%%%%%%%%%
%%%%%%%%%%%%%%%%%%%%%%%%%%%%%%%%%%%%%%%%%%%

\section{V-arithmetic pentagon graphs $\Gamma^{(6)}_1(14)$,
$\Gamma^{(6)}_2(14)$,
$\Gamma^{(6)}_3(14)$, $\Gamma^{(7)}_1(14)$, $\Gamma^{(7)}_2(14)$
and their fields}\label{sec:pentgraphs}

Here we prove Theorem \ref{thpentgraphs}.

\subsection{Some general results.}
\label{subsec:genresults} We use the following general results
from \cite{Nik2}.

\begin{theorem} (\cite[Theorem 1.2.1]{Nik2}) Let $\bff$ be a totally
real algebraic number field, and let each embedding
$\sigma:\bff\to \br$ corresponds to an interval
$[a_\sigma,b_\sigma]$ in $\br$ where
$$
\prod_{\sigma }{\frac{b_\sigma-a_\sigma}{4}}<1.
$$
In addition, let the natural number $m$ and the intervals
$[s_1,t_1],\dots, [s_m,t_m]$ in $\br$ be fixed. Then there exists
a constant $N(s_i,t_i)$ such that, if $\alpha$ is a totally real
algebraic integer and if the following inequalities hold for the
embeddings $\tau:\bff(\alpha) \to \br$:
$$
s_i\le \tau(\alpha)\le t_i\ \ for\ \ \tau=\tau_1,\dots ,\tau_m,
$$
$$
a_{\tau | \bff}\le \tau(\alpha)\le b_{\tau | \bff}\ \ for\ \
\tau\not=\tau_1,\dots,\tau_m,
$$
then
$$
[\bff(\alpha):\bff]\le N(s_i,t_i).
$$
\label{th121}
\end{theorem}

\begin{theorem} (\cite[Theorem 1.2.2]{Nik2})
Under the conditions of Theorem \ref{th121}, $N(s_i,t_i)$ can be
taken to be $N(s_i,t_i)=N_0$,
 where $N_0$ is the least natural number
solution of the inequality
\begin{equation}
N_0M\ln{(1/R)} - M\ln{(N_0+1)}-\ln{B}\ge \ln{S}. \label{cond for
n}
\end{equation}

Here
\begin{equation}
M=[\bff : \bq],\ \ \ B=2\sqrt{|{\rm discr\ } \bff|}; \label{MB}
\end{equation}
\begin{equation}
R=\sqrt{\prod_\sigma {\frac{b_\sigma-a_\sigma}{4}}},\ \ \
S=\prod_{i=1}^{m}{\frac{2er_i}{b_{\sigma_i}-a_{\sigma_i}}}
\label{RS}
\end{equation}
where
\begin{equation}
\sigma_i=\tau_i|\bff,\ \ \  r_i=\max\{{|b_i-a_{\sigma_i}|,
|b_{\sigma_i}-a_i|}\}. \label{ri}
\end{equation}
\label{th122}
\end{theorem}

We note that the proof of Theorems \ref{th121} and \ref{th122}
uses a variant of Fekete's Theorem (1923) about existence of
non-zero integer polynomials of bounded degree which differ only
slightly from zero on appropriate intervals. See \cite[Theorem
1.1.1]{Nik2}.

Below we will apply these results in two cases which are very
similar to used in \cite[Sec. 5.5]{Nik6} and \cite{Nik2}.

\medskip

{\bf Case 1.} For a natural $l\ge 3$ we denote
$\bff_l=\bq(\cos{(2\pi/l)})$. We consider a totally real algebraic
number field $\bk$ where $\bff_l\subset \bk=\bq(\alpha)$, and the
algebraic integer $\alpha$ satisfies
\begin{equation}
0<\sigma(\alpha)<a\sigma(\sin^2{(\pi/l)})
\label{case1.cond1sigma}
\end{equation}
for all $\sigma:\bk \to \br$ such that $\sigma\not=\sigma^{(+)}$, and
\begin{equation}
b_1<\sigma^{(+)}(\alpha)<b_2
\label{case1.cond1sigmapl}
\end{equation}
where $\sigma^{(+)}:\bk\to \br$ is the identity.  We assume that
$0<a<4$ and $b_1<b_2$. We denote $b=\max\{|b_1|, |b_2|\}$, and we
assume that $a\le b$. We want to estimate $[\bk:\bff_l]=N_0$ and
$N=[\bk:\bq]=N_0[\bff_l:\bq]$ from above.

For $l\ge 3$, we have $[\bff_l:\bq]=\varphi(l)/2$ where $\varphi(l)$ is 
the Euler function, and
$N_{\bff_l/\bq}(\sin^2{(\pi/l)})=\gamma(l)/4^{\varphi(l)/2}$ where
\begin{equation}
N_{\bff_l/\bq }(4\sin^2{(\pi/l)})=\gamma(l)=\left\{
\begin{array}{cl}
p &\ {\rm if}\  l=p^t>2\ {\rm where }\ p\ {\rm is\ prime,} \\
1 &\ {\rm otherwise.}
\end{array}\
\label{normsin1}
 \right .
\end{equation}
We have
$$
\frac{ba^N|N_{\bk/\bq}(\sin^2{(\pi/l)})|}{a\sin^2{(\pi/l)}}>
|N_{\bk/\bq}(\alpha)|\ge 1
$$
and
$$
\frac{b(a/4)^N\gamma(l)^{2N/\varphi(l)}}{a\sin^2{(\pi/l)}}>1.
$$
Equivalently, we have
\begin{equation}
N\left(\ln{\frac{2}{\sqrt{a}}-\frac{\ln{\gamma(l)}}{\varphi(l)}}\right)
<\ln{\sqrt{\frac{b}{a}}}-\ln{\sin{\frac{\pi}{l}}}\ \ \text{and}\ \
(\varphi(l)/2)|N.
\label{case1.1}
\end{equation}
Since $\gamma (l)\le l$, $\varphi(l)\ge Cl/\ln(\ln{l})$ for $l\ge
6$ where $C=\varphi(6)\ln{(\ln{6})}/6\ge 0.194399$,
$\sin{(\pi/l)}\le \pi/l$ for $l\ge 3$, there exists only finite
number of $l\ge 3$ such that \eqref{case1.1} has solutions $N\in \bn$.

More exactly, there exists only finite number of {\it exceptional
$l\ge 3$} such that
\begin{equation}
\ln{\frac{2}{\sqrt{a}}-\frac{\ln{\gamma(l)}}{\varphi(l)}}\le 0.
\label{case1.2}
\end{equation}
All non-exceptional $l$ satisfy the inequality
\begin{equation}
(\varphi(l)/2)\left(\ln{\frac{2}{\sqrt{a}}-
\frac{\ln{\gamma(l)}}{\varphi(l)}}\right)
<\ln{\sqrt{\frac{b}{a}}}-\ln{\sin{\frac{\pi}{l}}}.
\label{case1.3}
\end{equation}
Remark that exceptional $l$ also satisfy this inequality.

If $\gamma(l)=1$, \eqref{case1.3} implies that $l$ satisfies the
inequality
\begin{equation}
 {(C/2)\ln{(2/\sqrt{a})}}\,l<{\left(\ln{l}+\ln{(\sqrt{(b/a)}/\pi)}\right)
\ln{\ln{l}}}\,.
\label{case1.4}
\end{equation}
It follows that
\begin{equation}
l<L_0 \label{case1.5}
\end{equation}
where $L_0>3$ satisfies
\begin{equation}
{(C/2)\ln{(2/\sqrt{a})}}\,L_0\ge
{\left(\ln{L_0}+\ln{(\sqrt{(b/a)}/\pi)}\right) \ln{\ln{L_0}}}\,.
\label{case1.6}
\end{equation}
If $l=p^t$ where $p$ is prime, \eqref{case1.3} implies that $l$
satisfies the inequality
\begin{equation}
(C/2)\Delta(a)\,l<{\left(\ln{l}+\ln{(\sqrt{(b/a)}/\pi)}\right)
\ln{\ln{l}}}\,
\label{case1.7}
\end{equation}
where
\begin{equation}
\Delta(a)=\min_{l=p^t\ge L_0}
\left\{\ln{\frac{2}{\sqrt{a}}-\frac{\ln{\gamma(l)}}{\varphi(l)}>0}\right\}.
\label{case1.8}
\end{equation}
It follows that
\begin{equation} l<L_1
\end{equation}
where $L_1\ge L_0$ is a solution of the inequality
\begin{equation}
(C/2)\Delta(a)\,L_1\ge
{\left(\ln{L_1}+\ln{(\sqrt{(b/a)}/\pi)}\right)
\ln{\ln{L_1}}}\,.
\label{case1.9}
\end{equation}
Thus, to find all non-exceptional $l$ satisfying \eqref{case1.3},
we should check \eqref{case1.3} for all $l$ such that $3\le l<
L_1$, moreover, if $L_0\le l<L_1$, we can assume that $l=p^t$.
Their number is finite, and all of them can be effectively found.

For non-exceptional $l$ satisfying \eqref{case1.3}, we obtain
bounds
\begin{equation}
N_0=[\bk:\bff_l]\le \left[
\frac{\ln{\sqrt{b/a}}-\ln{\sin{(\pi/l)}}}
{(\varphi(l)/2)\left(\ln{(2/\sqrt{a})}-(\ln{\gamma(l)})/\varphi(l)\right)}
\right]
\label{case1.10}
\end{equation}
and
\begin{equation}
N=[\bk:\bq]\le \left[ \frac{\ln{\sqrt{b/a}}-\ln{\sin{(\pi/l)}}}
{(\varphi(l)/2)\left(\ln{(2/\sqrt{a})}-(\ln{\gamma(l)})/\varphi(l)\right)}
\right]\cdot(\varphi(l)/2)\,.
\label{case1.11}
\end{equation}
This using of the norm, we call the {\it Method B} (like in
\cite[Sec. 5.5]{Nik6}).

On the other hand, for fixed $l$, we obtain a bound for $N_0$
using Theorems \ref{th121}, \ref{th122} applied to $\bff=\bff_l$ and 
$\alpha$. We can take
\begin{equation}
R=\sqrt{|N_{\bff_l/\bq}(\sin^2{(\pi/l)})|(a/4)^{\varphi(l)/2}}=
\left(\frac{\gamma(l)^{1/\varphi(l)}a^{1/2}}{4}\right)^{\varphi(l)/2},
\label{case1.12}
\end{equation}
where
\begin{equation}
R<1\ \text{ if and only if }\
\ln{\frac{4}{\sqrt{a}}}-\frac{\ln{\gamma(l)}}{\varphi(l)}>0\,,
\label{case1.12a}
\end{equation}
\begin{equation}
M=[\bff_l:\bq]=\frac{\varphi(l)}{2},\ \
B=2\sqrt{|\text{discr}\,\bff_l|}
\label{case1.13}
\end{equation}
where the discriminant $|\text{discr}\,\bff_l|$ is given in
\eqref{discrFl}, and
\begin{equation}
S=\frac{2e\max\{a,b_2,a-b_1\}}{a\sin^2{(\pi/l)}}.
\label{case1.14}
\end{equation}
Then $[\bk:\bff_l]\le n_0$ and $[\bk:\bq]\le n_0\varphi(l)/2$
where $n_0$ is the least natural solution of the inequality
\eqref{cond for n}
\begin{equation} n_0M\ln{(1/R)} -
M\ln{(n_0+1)}-\ln{B}\ge \ln{S}.
\label{case1.15}
\end{equation}
In particular, this gives a bound for $[\bk:\bq]$ for exceptional
$l$ satisfying \eqref{case1.12a} and improves the bound
\eqref{case1.10} for $N_0$ when it is poor, which also improves
the bound for $[\bk:\bq]$. This using of Theorems \ref{th121},
\ref{th122}, we call the {\it Method A} (like in \cite[Sec.
5.5]{Nik6}).

We shall apply these Methods A and B to $\Gamma_2^{(6)}(14)$ in Sect.
\ref{subsubsec:Gamma6.2}.

\medskip

%%%%%%%%%%%%%%%%%%%%%%%%%%
%%%%%%%%%%%%%%%%%%%%%%%%%%

{\bf Case 2.} For natural $k\ge s\ge 3$, we denote
$\bff_{k,s}=\bq(\cos{(2\pi/k)},\,\cos{(2\pi/s)})$. We consider a
totally real algebraic number field $\bk$ where $\bff_{k,s}\subset
\bk=\bq(\alpha)$, and the algebraic integer $\alpha$ satisfies
\begin{equation}
0<\sigma(\alpha)<a\sigma(\sin^2{(\pi/k)}\sin^2{(\pi/s)})
\label{case2.cond1sigma}
\end{equation}
for all $\sigma:\bk \to \br$ such that $\sigma\not=\sigma^{(+)}$, and
\begin{equation}
b_1<\sigma^{(+)}(\alpha)<b_2
\label{case2.cond1sigmapl}
\end{equation}
where $\sigma^{(+)}:\bk\to \br$ is the identity.  We assume that
$0<a<16$ and $b_1<b_2$. We denote $b=\max\{|b_1|, |b_2|\}$, and we
assume that $a\le b$. We want to estimate $[\bk:\bff_{k,s}]=N_0$
and $N=[\bk:\bq]=N_0[\bff_{k,s}:\bq]$ for non-exceptional $k$ and
$s$ where {\it $l\ge 3$ is called exceptional} if
\begin{equation}
\ln{\frac{4}{\sqrt{a}}}-\frac{\ln{\gamma(l)}}{\varphi(l)}\le 0.
\label{case2.excepl}
\end{equation}
Equivalently, we have $4/\sqrt{a}\le \gamma(l)^{1/\varphi(l)}$.
We also assume that $k\ge s\ge s_0\ge 3$ where $s_0\ge 3$ is fixed.

We have $[\bff_{k,s}:\bq]=\varphi([k,s])/2\rho(k,s)$ where
$\rho(k,s)=1$ or $2$ is given in \eqref{defrho}, and
$N_{\bff_l/\bq}(\sin^2{(\pi/l)})=\gamma(l)/4^{\varphi(l)/2}$ where
$\gamma(l)$ is given in \eqref{normsin1}. We have
$$
\frac{ba^N|N_{\bk/\bq}(\sin^2{(\pi/k)})|
|N_{\bk/\bq}(\sin^2{(\pi/s)})| }
{a\sin^2{(\pi/k)}\sin^2{(\pi/s)}}> |N_{\bk/\bq}(\alpha)|\ge 1
$$
and
$$
\frac{b(a/16)^N\gamma(k)^{2N/\varphi(k)}\gamma(s)^{2N/\varphi(s)}}
{a\sin^2{(\pi/k)}\sin^2{(\pi/s)}}>1.
$$
Equivalently, we obtain
\begin{equation}
N\left(\ln{\frac{4}{\sqrt{a}}-\frac{\ln{\gamma(k)}}{\varphi(k)}-
\frac{\ln{\gamma(s)}}{\varphi(s)} }\right)
<\ln{\sqrt{\frac{b}{a}}}-\ln{\sin{\frac{\pi}{k}}}-
\ln{\sin{\frac{\pi}{s}}}\ \text{and}\
\frac{\varphi([k,s])}{2\rho(k,s)}\left|\right. N.
\label{case2.1}
\end{equation}
Since $\gamma (l)\le l$, $\varphi(l)\ge Cl/\ln(\ln{l})$ for $l\ge
6$ where $C=\varphi(6)\ln{(\ln{6})}/6$, \linebreak 
$\sin{(\pi/l)}\le \pi/l$
for $l\ge 3$, there exists only finite number of pairs $(k,s)$ such
that \eqref{case2.1} has solutions $N\in \bn$
where both $k$ and $s$ are non-exceptional.

More exactly, there exists only finite number of {\it exceptional
pairs $(k,s)$} where a pair $(k,s)$ (consisting of non-exceptional
$k$ and $s$) is called exceptional  if
\begin{equation}
\ln{\frac{4}{\sqrt{a}}}-\frac{\ln{\gamma(k)}}{\varphi(k)}-
\frac{\ln{\gamma(s)}}{\varphi(s)}\le 0.
\label{case2.2}
\end{equation}
All non-exceptional pairs $(k,s)$ satisfying \eqref{case2.1}
satisfy the inequality
\begin{equation}
\frac{\varphi([k,s])}{2\rho(k,s)}\cdot \left(
\ln{\frac{4}{\sqrt{a}}}-\frac{\ln{\gamma(k)}}{\varphi(k)}-
\frac{\ln{\gamma(s)}}{\varphi(s)}\right) <
\ln{\sqrt{\frac{b}{a}}}-\ln{\sin{\frac{\pi}{k}}}-
\ln{\sin{\frac{\pi}{s}}}\ .
\label{case2.3}
\end{equation}
Remark that exceptional pairs $(k,s)$ also satisfy this inequality.

If $\gamma(k)=\gamma(s)=1$ and $k\ge s$, \eqref{case2.3} implies
\begin{equation}
{(C/2)\ln{(4/\sqrt{a})}}
k<{\left(2\ln{k}+\ln{(\sqrt{(b/a)}/\pi^2)}\right)\ln{\ln{k}}}.
\label{case2.4}
\end{equation}
It follows that
\begin{equation}
s_0\le s\le k< K_0
\label{case2.5}
\end{equation}
where $K_0>3$ satisfies
\begin{equation}
{(C/2)\ln{(4/\sqrt{a})}}
K_0\ge {\left(2\ln{K_0}+\ln{(\sqrt{(b/a)}/\pi^2)}\right)\ln{\ln{K_0}}}.
\label{case2.6}
\end{equation}
If one of $\gamma(k)$, $\gamma(s)$ is not equal to $1$, then
\eqref{case2.3} implies for non-exceptional pairs $(k,s)$ that
\begin{equation}
{(C/2)\Delta_1(a)}
k<{\left(2\ln{k}+\ln{(\sqrt{(b/a)}/\pi^2)}\right)\ln{\ln{k}}}
\label{case2.7}
\end{equation}
where
\begin{equation}
\Delta_1(a)=\min_{k\ge s\ge s_0, k\ge K_0}\left\{\ln{\frac{4}{\sqrt{a}}}-
\frac{\ln{\gamma(s)}}{\varphi(s)}-
\frac{\ln{\gamma(k)}}{\varphi(k)}\,>\,0\right\}\,.
\label{case2.8}
\end{equation}
It follows that
\begin{equation}
s_0\le s\le k<K_1
\label{case2.9}
\end{equation}
where $K_1\ge K_0$ is a solution of the inequality
\begin{equation}
{(C/2)\Delta_1(a)}
K_1\ge {\left(2\ln{K_1}+\ln{(\sqrt{(b/a)}/\pi^2)}\right)\ln{\ln{K_1}}}.
\label{case2.10}
\end{equation}
Thus, to find all non-exceptional pairs $(k,s)$ satisfying \eqref{case2.3},
we should check \eqref{case2.3} for all $s_0\le s\le k<K_1$; moreover, if
$K_0\le k\le K_1$, we can assume that one of $k$ and $s$ is equal to $p^t$
where $p$ is prime. The number of such pairs is finite, and all of them 
can be effectively found.

For such non-exceptional pairs $(k,s)$ satisfying \eqref{case2.3},
we obtain bounds
\begin{equation}
N_0=[\bk:\bff_{k,s}]\le \left[\frac
{\ln{\sqrt{\frac{b}{a}}}-\ln{\sin{\frac{\pi}{k}}}-
\ln{\sin{\frac{\pi}{s}}}} {\frac{\varphi([k,s])}{2\rho(k,s)}\cdot
\left( \ln{\frac{4}{\sqrt{a}}}-\frac{\ln{\gamma(k)}}{\varphi(k)}-
\frac{\ln{\gamma(s)}}{\varphi(s)}\right)}\right]
\label{case2.11}
\end{equation}
and
\begin{equation}
N=[\bk:\bq]\le \left[\frac
{\ln{\sqrt{\frac{b}{a}}}-\ln{\sin{\frac{\pi}{k}}}-
\ln{\sin{\frac{\pi}{s}}}} {\frac{\varphi([k,s])}{2\rho(k,s)}\cdot
\left( \ln{\frac{4}{\sqrt{a}}}-\frac{\ln{\gamma(k)}}{\varphi(k)}-
\frac{\ln{\gamma(s)}}{\varphi(s)}\right)}\right]\cdot
\frac{\varphi([k,s])}{2\rho(k,s)}\ .
\label{case2.12}
\end{equation}
This using of the norm, we call the {\it Method B} (like in
\cite[Sec. 5.5]{Nik6}).

On the other hand, for a fixed pair $(k,s)$, we can obtain a bound
for $N_0$ using Theorems \ref{th121}, \ref{th122} 
applied to $\bff=\bff_{k,s}$ and $\alpha$. We can 
take 
\begin{equation}
R=\sqrt{|N_{\bff_{k,s}/\bq}(\sin^2{\frac{\pi}{k}}
\sin^2{\frac{\pi}{s}})|
\left(\frac{a}{4}\right)^{\frac{\varphi([k,s])}{2\rho(k,s)}}}=
\left(\frac{\gamma(k)^{\frac{1}{\varphi(k)}}
\gamma(s)^{\frac{1}{\varphi(s)}}
a^{\frac{1}{2}}}{8}\right)^{\frac{\varphi([k,s])}{2\rho(k,s)}}
\label{case2.13}
\end{equation}
where
\begin{equation}
R<1\ \text{if and only if}\
\ln{\frac{8}{\sqrt{a}}}-\frac{\ln{\gamma(k)}}{\varphi(k)}-
\frac{\ln{\gamma(s)}}{\varphi(s)}>0,
\label{case2.14}
\end{equation}
\begin{equation}
M=[\bff_{k,s}:\bq]=\frac{\varphi([k,s])}{2\rho(k,s)},\ \
B=2\sqrt{|\text{discr}\,\bff_{k,s}|}
\label{case2.15}
\end{equation}
where the discriminant $|\text{discr}\,\bff_{k,s}|$ is given in
\eqref{discrFks1} and \eqref{discrFks2}, and
\begin{equation}
S=\frac{2e\max\{a,b_2,a-b_1\}}{a\sin^2{(\pi/s)}\sin^2{(\pi/k)}}.
\label{case2.16}
\end{equation}
For all pairs $(k,s)$ satisfying \eqref{case2.14}, we obtain the
bounds $[\bk:\bff_{k,s}]\le n_0$ and $[\bk:\bq]\le
n_0\varphi([k,s])/(2\rho(k,s))$ where $n_0$ is the least natural
solution of the inequality \eqref{cond for n},
$$
n_0M\ln{(1/R)} - M\ln{(n_0+1)}-\ln{B}\ge \ln{S}.
$$
For $a<16$ and $k,s\ge 3$, all
pairs $(k,s)$, except finite number, satisfy \eqref{case2.14},
and we can apply this method to all these pairs.
In particular, this gives a bound for $[\bk:\bq]$ for all exceptional pairs
$(k,s)$ satisfying \eqref{case2.14}, and it improves the bound
\eqref{case2.11} for $N_0$ when it is poor,
which also improves the bound for $[\bk:\bq]$.
This using of Theorems \ref{th121}, \ref{th122}, we call the {\it
Method A} (like in \cite[Sec. 5.5]{Nik6}).

We apply these Methods A and B to $\Gamma_1^{(6)}(14)$ in Sect.
\ref{subsec:pentgamma61} and $\Gamma_3^{(6)}(14)$,
$\Gamma_1^{(7)}(14)$, $\Gamma_2^{(7)}(14)$ in Sect.
\ref{subsec:pentgraphs}.

\subsection{V-arithmetic pentagon graphs $\Gamma^{(6)}_1(14)$
and their fields.}
\label{subsec:pentgamma61}

Here we consider
V-arithmetic 3-dimensional graphs $\Gamma^{(6)}_1(14)$ and their 
fields. See  Definition \ref{defpentgamma61fields} and Figure 
\ref{pentgraph61}.  

First, let us consider a pentagon $\Delta$ on hyperbolic plane which has all
angles right ($=\pi/2$) except one angle which is equal to $\pi/k$,
$k\ge 2$. We denote $c=2\cos(\pi/k)$ where $-2< c< 2$.
When $c=0$, we obtain a pentagon with right angles.

Let
$$
P(\Delta)=\{\beta_1,\beta_2,\beta_3,\beta_4,\beta_5\}
$$
correspond to five consecutive sides of $\Delta$ and $\beta_1,\beta_2$ are
perpendicular to the vertex of $\Delta$ with the angle $\pi/k$. Thus,
$\beta_1\cdot\beta_2=c$ and $\beta_i\cdot \beta_{i+1}=0$ for $2\le i\le 5$.
The Gram graph of $\Delta$ is given in Figure \ref{pentagon}
where we denote $b_{ij}=\beta_i\cdot \beta_j$ with $b_{ij}>2$.

\begin{figure}
\begin{center}
\includegraphics[width=11cm]{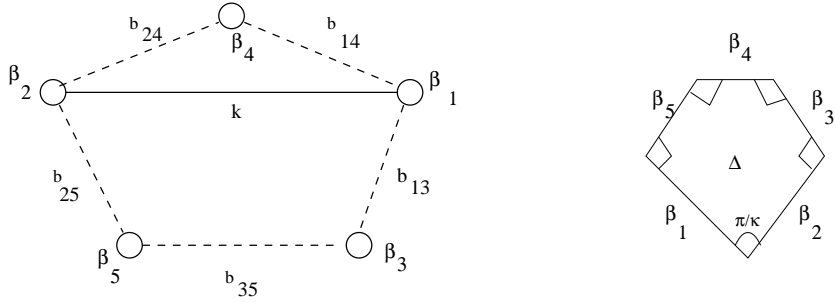}
\end{center}
\caption{The graph of a pentagon with one non-right angle}
\label{pentagon}
\end{figure}

The vectors $\beta_1,\beta_2,\beta_4$ generate the form $\Phi$ which defines
the hyperbolic plane. The determinant of their Gram matrix is equal to
\begin{equation}
d_{124}=-8+2cb_{14}b_{24}+2b_{14}^2+2b_{24}^2+2c^2.
\label{d124}
\end{equation}
It must be positive for the geometric embedding $\sigma^{(+)}$, and
it must be negative for non-geometric embeddings $\sigma\not=\sigma^{(+)}$.
Thus,
\begin{equation}
0<b_{14}^2+b_{24}^2+cb_{14}b_{24}<4-c^2
\label{d124notgeom}
\end{equation}
for $\sigma\not=\sigma^{(+)}$. Here the first inequality follows from
$-2<c<2$ for any embedding since $c=2\cos(\pi/k)$ where $k\ge 2$ is
an integer.

Applying this to the pentagon $\Delta$ corresponding to
$\Gamma^{(6)}_1$, we obtain $\beta_i=\delta_i$ for $i\not=4$, and
$\beta_4=\widetilde{\delta}_4=(\delta_4+\cos{(\pi/m)}e)/\sin{(\pi/m)}$.
Then $b_{14}=a_{14}/\sin{(\pi/m)}$, $b_{24}=a_{24}/\sin{(\pi/m)}$
and for
\begin{equation}
\alpha=a_{14}^2+a_{24}^2+2\cos{(\pi/k)}a_{14}a_{24}
\label{alphagam16}
\end{equation}
we obtain
\begin{equation}
0<\sigma(\alpha)< 4\sigma(\sin^2{\frac{\pi}{k}}
\sin^2{\frac{\pi}{m}}) \label{d124anotgeom}
\end{equation}
for $\sigma\not=\sigma^{(+)}$, and
\begin{equation}
12<\sigma^{(+)}(\alpha)<(2\cdot 14)^2=28^2. \label{d124ageom}
\end{equation}
since $2<\sigma^{(+)}(a_{ij})<14$ and
$\sigma^{(+)}(2\cos{(\pi/k)})\ge 1$ for $k\ge 3$.

Thus, $\alpha$ is a totally positive algebraic integer. By
\eqref{d124anotgeom} and \eqref{d124ageom}, the ground field
$\bk=\bq(\alpha)$ is generated by $\alpha$, and
$\bff_{k,m}=\bq(\cos(2\pi/k),\cos(2\pi/m))\subset \bk$. We get
exactly the same situation which we had considered in \cite[Sec.
5.5]{Nik6}. Only for the inequality \eqref{d124ageom} we had
$<14^2$ in \cite{Nik6} instead of $28^2$ here. Then exactly the
same considerations as in \cite[Sec. 5.5]{Nik6} show that
$[\bk:\bq]\le 56$. The worst case is achieved for
$\{k,m\}=\{3,113\}$.

\medskip

More exactly, we apply the methods A and B of Case 2 in Sec.
\ref{subsec:genresults} to $a=4$, $b_1=12$, $b_2=28^2$ (then
$b=28^2$) and $k:=\max\{k,m\}$, $s:=\min\{k,m\}$ where $k\ge s\ge
3$ and $s_0=3$.

At first, we apply the Method B. Since $\ln{2}>\ln{(3)}/2$, all
$k$ and $s$ are non-exceptional, see \eqref{case2.excepl}. All
exceptional pairs $(k,s)$ that is when \eqref{case2.2} which is
\begin{equation} \ln{2}-\frac{\ln{\gamma(k)}}{\varphi(k)}-
\frac{\ln{\gamma(s)}}{\varphi(s)}\le 0
\label{Gamma1.6.1}
\end{equation}
satisfies,
 are $(k,s=3)$ where
$k=3,4,5,7,8,9,11,13,17,19$; $(k,s=4)$ where $k=4,5$; $(k,s=5)$
where $k=5,7$.

We can take $K_0=306$ in \eqref{case2.6}. Then
$\Delta_1(4)=\ln(2)-\ln(3)/2-\ln(307)/306\ge 0.1251$, and
$K_1=2760$ can be taken in \eqref{case2.10}.
Checking \eqref{case2.3} for $3\le s\le
k<2760$, we obtain that $3\le s\le 90$ and $3\le s\le k\le 420$.
Moreover, $k\le 90$ for $11\le s\le 90$. For all these pairs
$(k,s)$ satisfying \eqref{case2.3} which is
\begin{equation}
\frac{\varphi([k,s])}{2\rho(k,s)}\cdot \left(
\ln{2}-\frac{\ln{\gamma(k)}}{\varphi(k)}-
\frac{\ln{\gamma(s)}}{\varphi(s)}\right) <
\ln{14}-\ln{\sin{\frac{\pi}{k}}}- \ln{\sin{\frac{\pi}{s}}}\ ,
\label{Gamma1.6.2}
\end{equation}
we obtain
\begin{equation}
[\bff_{k,s}:\bq]=\frac{\varphi([k,s])}{2\rho(k,s)}\le 56
\label{Gamma1.6.3}
\end{equation}
where $56$ is achieved for $(k,s)=(113,3)$. Moreover, for all
these non-exceptional pairs $(k,s)$ we obtain the bound
\eqref{case2.11} which is
\begin{equation}
N_0=[\bk:\bff_{k,s}]\le \left[\frac
{\ln{14}-\ln{\sin{\frac{\pi}{k}}}- \ln{\sin{\frac{\pi}{s}}}}
{\frac{\varphi([k,s])}{2\rho(k,s)}\cdot \left(
\ln{2}-\frac{\ln{\gamma(k)}}{\varphi(k)}-
\frac{\ln{\gamma(s)}}{\varphi(s)}\right)}\right],
\label{Gamma1.6.4}
\end{equation}
and finally we obtain the bound \eqref{case2.12} which is
\begin{equation}
N=[\bk:\bq]\le \left[\frac {\ln{14}-\ln{\sin{\frac{\pi}{k}}}-
\ln{\sin{\frac{\pi}{s}}}} {\frac{\varphi([k,s])}{2\rho(k,s)}\cdot
\left( \ln{2}-\frac{\ln{\gamma(k)}}{\varphi(k)}-
\frac{\ln{\gamma(s)}}{\varphi(s)}\right)}\right]\cdot
\frac{\varphi([k,s])}{2\rho(k,s)}\ .
\label{Gamma1.6.5}
\end{equation}
If either a pair $(k,s)$ is exceptional, or the right hand side of
\eqref{Gamma1.6.5} is more than $56$ (these are possible
only for pairs $(k,s)$ with $3\le s\le 7$ and $s\le k\le 420$),
we also apply to
the pair $(k,s)$ the method A of the Case 2 to improve the poor bound
\eqref{Gamma1.6.4} of $N_0=[\bk:\bff_{k,s}]$ for non-exceptional $(k,s)$.
We can apply this method to any pair $(k,s)$ since \eqref{case2.14} is
valid for all $k\ge s\ge 3$ if $a=4$. Finally we obtain that
$[\bk:\bq]\le 56$.

\subsection{V-arithmetic pentagon graphs $\Gamma^{(6)}_2(14)$,
$\Gamma^{(6)}_3(14)$,\\
$\Gamma^{(7)}_1(14)$, $\Gamma^{(7)}_2(14)$, and their ground fields}
\label{subsec:pentgraphs}

\begin{figure}
\begin{center}
\includegraphics[width=11cm]{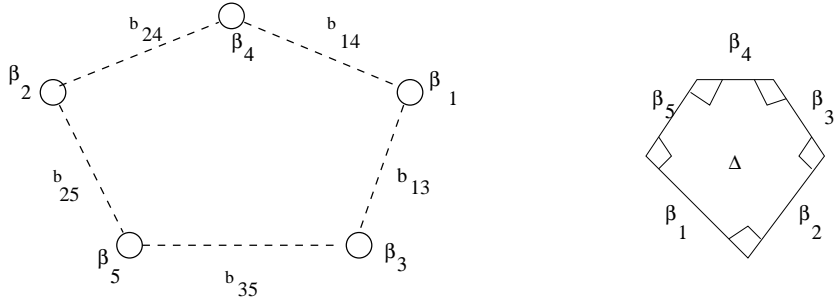}
\end{center}
\caption{The graph of pentagon with right angles}
\label{pentagon0}
\end{figure}

Here we consider V-arithmetic 3 and 4-dimensional graphs
$\Gamma^{(6)}_2(14)$, $\Gamma^{(6)}_3(14)$, $\Gamma^{(7)}_1(14)$,
$\Gamma^{(7)}_2(14)$ and their fields. See Definition 
\ref{defpentagonfields} and Figure \ref{pentgraphs}. 

They are related to pentagons with right angles. This can be
considered as a specialisation of the previous case when $c=0$.
See Figure \ref{pentagon0}.

The Gram matrix of any 4 elements $\beta_i$ must have zero determinant.
Considering this for all four elements $\beta_i$, we obtain equations
\begin{equation}
\begin{array}{l}
4b_{14}^2=(4-b_{13}^2)(4-b_{24}^2),\\
4b_{24}^2=(4-b_{14}^2)(4-b_{25}^2),\\
4b_{25}^2=(4-b_{24}^2)(4-b_{35}^2),\\
4b_{35}^2=(4-b_{25}^2)(4-b_{13}^2),\\
4b_{13}^2=(4-b_{35}^2)(4-b_{14}^2).
\end{array}
\label{equationspent}
\end{equation}
It follows that
\begin{equation}
\gamma=b_{13}b_{14}b_{24}b_{25}b_{35}=
-(4-b_{13}^2)(4-b_{14}^2)(4-b_{24}^2)(4-b_{25}^2)(4-b_{35}^2)/2^5
\label{defgamma}
\end{equation}
since for the geometric embedding the sign must be positive.
One can also see the sign from the determinant of the Gram matrix of
$\beta_1,\dots,\beta_5$ which is equal to zero:
$$
8b_{13}^2+8b_{35}^2+8b_{25}^2+8b_{24}^2+8b_{14}^2+
2b_{13}b_{14}b_{24}b_{25}b_{35}-
$$
$$
-2b_{13}^2b_{24}^2-2b_{13}^2b_{25}^2-
2b_{14}^2b_{25}^2-2b_{14}^2b_{35}^2-2b_{24}^2b_{35}^2-32 = 0.
$$
Thus, \eqref{equationspent} and \eqref{defgamma} give equations of
a hyperbolic pentagon with right angles.

For the geometric embedding $\sigma^{(+)}$ the expression $\gamma$
is positive, and for $\sigma\not=\sigma^{(+)}$ it is negative
because $0<b_{ij}^2<4$ for all $b_{ij}$. We have the following
very important statement.

\begin{lemma} The minimum of $\gamma$ in \eqref{defgamma} for
$b_{ij}$ satisfying \eqref{equationspent} and $0\le b_{ij}^2\le 4$
is achieved for equal $b_{ij}^2$. Then it is equal to
$$
-(\sqrt{5}-1)^5=-2.8854381999983\dots\ .
$$
\label{mingamma}
\end{lemma}

\begin{proof} We denote $q_{ij}=4-a_{ij}^2$. Then equations
\eqref{equationspent} are
\begin{equation}
\begin{array}{l}
q_{14}=4-\frac{q_{13}q_{24}}{4},\\
q_{24}=4-\frac{q_{14}q_{25}}{4},\\
q_{25}=4-\frac{q_{24}q_{35}}{4},\\
q_{35}=4-\frac{q_{25}q_{13}}{4},\\
q_{13}=4-\frac{q_{35}q_{14}}{4}.
\end{array}
\label{equationspentq}
\end{equation}
We should find maximum of
$$
A=q_{13}q_{14}q_{24}q_{25}q_{35}
$$
in the region $0\le q_{ij}\le 4$ with the conditions
\eqref{equationspentq}. If one of $q_{ij}=0$ or $q_{ij}=4$, then
$A=0$. Thus, the maximum is taken when all $0<q_{ij}<4$.

We denote $x=q_{13}$ and $y=q_{24}$. Then $A=2^6F(x,y)$ where
$$
F(x,y)=\frac{x^2y^2+16xy-4x^2y-4xy^2}{16-xy},
$$
and $0<x<4$ and $0<y<4$ are free variables. We have
$$
\frac{\partial F}{\partial x}=
\frac{-x^2y^3+4x^2y^2+32y^2x-128xy-64y^2+256y}{(16-xy)^2}
$$
and
$$
\frac{\partial F}{\partial y}=
\frac{-x^3y^2+4x^2y^2+32x^2y-128xy-64x^2+256x}{(16-xy)^2}\,.
$$
From $\frac{\partial F}{\partial x}=0$, we get
$$
-x^2y^2+4x^2y+32yx-128x-64y+256=0.
$$
From  $\frac{\partial F}{\partial y}=0$, we get
$$
-x^2y^2+4xy^2+32xy-128y-64x+256=0.
$$
Taking difference of these equations, we get $(x-y)(xy-16)=0$
which gives $x=y$ and $(x^2+4x-16)(x-4)^2=0$. It follows that
$x=y=2(\sqrt{5}-1)$.

It follows the statement.
\end{proof}

By Lemma \ref{mingamma}, we have
\begin{equation}
-2.88543819\dots\ = -\gamma_0=-(\sqrt{5}-1)^5<\sigma(\gamma)<0.
\label{inegamma}
\end{equation}
where we denote
\begin{equation}
\gamma_0=(\sqrt{5}-1)^5=2.885438199983\dots\ \le 2.8854382.
\label{gamma0}
\end{equation}

For pentagons $\Delta$
of graphs $\Gamma^{(6)}_2$, $\Gamma^{(6)}_3$, $\Gamma^{(7)}_1$
and $\Gamma^{(7)}_2$, we must replace $\beta_i$ by $\delta_i$
or $\widetilde{\delta_i}$ given in the definition which gives $b_{ij}$
from $a_{ij}$. From these expressions, we get
\begin{equation}
b_{13}b_{14}b_{24}b_{25}b_{35}=\frac{a_{13}a_{14}a_{24}a_{25}a_{35}}
{\sin^2{\frac{\pi}{m}}}\ \ \ for \ \ \Gamma_2^{(6)},
\label{equ62}
\end{equation}

\begin{equation}
b_{13}b_{14}b_{24}b_{25}b_{35}=\frac
{a_{13}a_{14}a_{24}a_{25}a_{35}+
2\cos{\frac{\pi}{m_1}}\cos{\frac{\pi}{m_3}}a_{14}a_{24}a_{25}a_{35}}
{\sin^2{\frac{\pi}{m_1}}{\sin^2{\frac{\pi}{m_3}}}}\ \ \ for \ \ \Gamma_3^{(6)},
\label{equ63}
\end{equation}

\begin{equation}
b_{13}b_{14}b_{24}b_{25}b_{35}=\frac{a_{13}a_{14}a_{24}a_{25}a_{35}}
{\sin^2{\frac{\pi}{m_1}}\sin^2{\frac{\pi}{m_2}} }\ \ \ for \ \ \Gamma_1^{(7)},
\label{equ71}
\end{equation}

\begin{equation}
b_{13}b_{14}b_{24}b_{25}b_{35}=\frac
{a_{13}a_{14}a_{24}a_{25}a_{35}+
2\cos{\frac{\pi}{m_1}}\cos{\frac{\pi}{m_3}}a_{14}a_{24}a_{25}a_{35}}
{\sin^2{\frac{\pi}{m_1}}{\sin^2{\frac{\pi}{m_3}}}
{\sin^2{\frac{\pi}{m_2}}}}\ \ \ for \ \ \Gamma_2^{(7)}.
\label{equ72}
\end{equation}

We consider an algebraic integer $\alpha\in \bk$ which is equal to
\begin{equation}
\alpha=\left\{
\begin{array}{ll}
a_{13}a_{14}a_{24}a_{25}a_{35} & \mbox{for
$\Gamma_2^{(6)}$ and $\Gamma_1^{(7)}$} \\
2a_{13}a_{14}a_{24}a_{25}a_{35}+ & \\
+4\cos{(\pi/m_1)}\cos{(\pi/m_3)}a_{14}a_{24}a_{25}a_{35} &
\mbox{for $\Gamma_3^{(6)}$ and $\Gamma_2^{(7)}$}
\end{array}.
\right.
\label{alpha}
\end{equation}
By \eqref{inegamma}, for $\sigma\not=\sigma^{(+)}$, we obtain
the inequalities
\begin{equation}
0>\sigma(\alpha)>-\gamma_0\sigma(\sin^2{\frac{\pi}{m}})\ \mbox{for
$\Gamma_2^{(6)}$}, \label{inegamma26}
\end{equation}

\begin{equation}
0>\sigma(\alpha)>-2\gamma_0\sigma(\sin^2{\frac{\pi}{m_1}}
\sin^2{\frac{\pi}{m_3}})\
\mbox{for $\Gamma_3^{(6)}$}, \label{inegamma36}
\end{equation}

\begin{equation}
0>\sigma(\alpha)>-\gamma_0\sigma(\sin^2{\frac{\pi}{m_1}}
\sin^2{\frac{\pi}{m_2}})\
\mbox{for $\Gamma_1^{(7)}$}, \label{inegamma17}
\end{equation}

\begin{equation}
0>\sigma(\alpha)>-2\gamma_0\sigma(\sin^2{\frac{\pi}{m_1}}
\sin^2{\frac{\pi}{m_3}}
\sin^2{\frac{\pi}{m_2}})\ \mbox{for $\Gamma_2^{(7)}$}.
\label{inegamma27}
\end{equation}
Since we consider 14-minimal graphs, we have
\begin{equation}
2^5<\sigma^{(+)}(\alpha)<14^5\
\mbox{for $\Gamma_2^{(6)}(14)$, $\Gamma_1^{(7)}(14)$},
\label{inegamma+2617}
\end{equation}
\begin{equation}
2^6<\sigma^{(+)}(\alpha)<2\cdot 14^5+4\cdot 14^4=32\cdot 14^4\
\mbox{for $\Gamma_3^{(6)}(14)$, $\Gamma_2^{(7)}(14)$},
\label{inegamma+3627}
\end{equation}

It follows that $\bk=\bq(\alpha)$. Since
$2\gamma_0=5.7708763999663\dots\ <16$, all these cases are similar
to considered in \cite[Sec. 5.5]{Nik6} (and originally in
\cite{Nik2}). We have adapted these considerations to our case
here in Sec. \ref{subsec:genresults}, and we have applied them in
Sec. \ref{subsec:pentgamma61}. For each graph, we give details
below which will be very important for further study.

\subsubsection{V-arithmetic pentagon graphs $\Gamma^{(6)}_2(14)$.}
\label{subsubsec:Gamma6.2}
For $\Gamma^{(6)}_2(14)$, considering $(-\alpha)$ from \eqref{alpha},
we apply the methods A and B of Case 1 in Sec. \ref{subsec:genresults}
to $a=\gamma_0=(\sqrt{5}-1)^5\le 2.8854382$, $b_1=-14^5$,
$b_2=-2^5$ (then $b=14^5$), and $l:= m$.

At first, we apply the Method B. All
exceptional $l\ge 3$ that is when \eqref{case1.2} which is
\begin{equation}
\ln{\frac{2}{\sqrt{\gamma_0}}-\frac{\ln{\gamma(l)}}{\varphi(l)}}\le 0
\label{Gamma2.6.1}
\end{equation}
satisfies are $l=3,4,5,7,8,9,11,13,17,19$.

We can take $L_0=1540$ in \eqref{case1.6}. Then
$$
\Delta_1(\gamma_0)=
\ln(\frac{2}{\sqrt{\gamma_0}})-\frac{\ln{1543}}{1542}\ge 0.1585\,,
$$
and $L_1=1595$ can be taken in \eqref{case1.9}.
Checking \eqref{case1.3} for $3\le  l< 1595$,
we obtain that $3\le l\le 510$.
For all these $l$ such that \eqref{case1.3} which is
\begin{equation}
\frac{\varphi(l)}{2}\cdot \left(\ln{\frac{2}{\sqrt{\gamma_0}}-
\frac{\ln{\gamma(l)}}{\varphi(l)}}\right)
<\ln{\sqrt{\frac{14^5}{\gamma_0}}}-\ln{\sin{\frac{\pi}{l}}}
\label{Gamma2.6.2}
\end{equation}
satisfies, we obtain
\begin{equation}
[\bff_{l}:\bq]=\frac{\varphi(l)}{2}\le 75
\label{Gamma2.6.3}
\end{equation}
where $75$ is achieved for $l=151$. Moreover, for all
these non-exceptional $l$ we obtain the bound
\eqref{case1.10} which is

\begin{equation}
N_0=[\bk:\bff_l]\le \left[
\frac{\ln{\sqrt{\frac{14^5}{\gamma_0}}}-\ln{\sin{\frac{\pi}{l}}}}
{\frac{\varphi(l)}{2}\cdot \left(\ln{\frac{2}{\sqrt{\gamma_0}}-
\frac{\ln{\gamma(l)}}{\varphi(l)}}\right)          }
\right],
\label{Gamma2.6.4}
\end{equation}
and finally we obtain the bound \eqref{case1.11} which is
\begin{equation}
N=[\bk:\bq]\le 
\left[
\frac{\ln{\sqrt{\frac{14^5}{\gamma_0}}}-\ln{\sin{\frac{\pi}{l}}}}
{\frac{\varphi(l)}{2}\cdot \left(\ln{\frac{2}{\sqrt{\gamma_0}}-
\frac{\ln{\gamma(l)}}{\varphi(l)}}\right)          }
\right]\cdot 
\frac{\varphi(l)}{2}\ .
\label{Gamma2.6.5}
\end{equation}
If either $l$ is exceptional, or the right hand side of
\eqref{Gamma2.6.5} is more than $75$ (this is possible only
for $3\le l\le 83$), we also apply to $l$
the method A of the Case 1 to improve the poor bound
\eqref{Gamma2.6.4} for $N_0=[\bk:\bff_{l}]$ for non-exceptional $l$.
We can apply this method to any $l\ge 3$ since \eqref{case1.12a} is
valid for all $l\ge 3$ if $a=\gamma_0$. Finally we obtain that
$[\bk:\bq]\le 75$.

\subsubsection{V-arithmetic pentagon graphs $\Gamma^{(6)}_3(14)$.}
\label{subsubsec:Gamma6.3}
In this case, considering $(-\alpha)$ from \eqref{alpha},
we apply the methods A and B of Case 2 in Sec.
\ref{subsec:genresults} to $a=2\gamma_0\le 5.7708764$,
$b_1=-32\cdot 14^4$, $b_2=-2^6$ (then $b=32\cdot 14^4$) and
$k:=\max\{m_1,m_3\}$, $s:=\min\{m_1,m_3\}$ where $k\ge s\ge 3$ and
$s_0=3$.

At first, we apply the Method B. Exceptional $l\ge 3$ satisfy
\eqref{case2.excepl} which is
\begin{equation}
\ln{\frac{4}{\sqrt{2\gamma_0}}}-\frac{\ln{\gamma(l)}}{\varphi(l)}\le 0.
\label{Gamma3.6.0}
\end{equation}
It follows that $l=3$ is the only exceptional.

All exceptional pairs $(k,s)$ where $k\ge s \ge 4$
that is when \eqref{case2.2} which is
\begin{equation}
\ln{\frac{4}{\sqrt{2\gamma_0}}}-\frac{\ln{\gamma(k)}}{\varphi(k)}-
\frac{\ln{\gamma(s)}}{\varphi(s)}\le 0.
\label{Gamma3.6.1}
\end{equation}
satisfies, are $(k,s=4)$ where
$k=4,5,7,8,9,11,13,17,19$; $(k,s=5)$ where $k=5,7,8,9,11,13,17,19,23,
29,31$; $(k,s=7)$ where $k=7,11,13$.

We can take $K_0=630$ in \eqref{case2.6}. Then
$$
\Delta_1(2\gamma_0)=\ln{\frac{4}{\sqrt{2\gamma_0}}}-
\frac{\ln{5}}{4}-\frac{\ln{631}}{630} \ge 0.097289,
$$
and $K_1=4684$ can be taken in \eqref{case2.10}. Checking
\eqref{case2.3} for $4\le s\le k<4684$, we obtain that $4\le s\le
210$ and $4\le s\le k\le 870$. Moreover, $k\le 210$ for $14\le
s\le 210$. For all these pairs $(k,s)$ satisfying \eqref{case2.3}
which is
\begin{equation}
\frac{\varphi([k,s])}{2\rho(k,s)}\cdot \left(
\ln{\frac{4}{\sqrt{2\gamma_0}}}-\frac{\ln{\gamma(k)}}{\varphi(k)}-
\frac{\ln{\gamma(s)}}{\varphi(s)}\right) < \ln{\frac{4\cdot
14^2}{\sqrt{\gamma_0}}}-\ln{\sin{\frac{\pi}{k}}}-
\ln{\sin{\frac{\pi}{s}}}\, ,  \label{Gamma3.6.2}
\end{equation}
we obtain
\begin{equation}
[\bff_{k,s}:\bq]=\frac{\varphi([k,s])}{2\rho(k,s)}\le 138
\label{Gamma3.6.3}
\end{equation}
where $138$ is achieved for $(k,s)=(139,5)$. Moreover, for all
these \linebreak non-exceptional pairs $(k,s)$ we obtain the bound
\eqref{case2.11} which is
\begin{equation}
N_0=[\bk:\bff_{k,s}]\le \left[\frac {\ln{\frac{4\cdot
14^2}{\sqrt{\gamma_0}}}-\ln{\sin{\frac{\pi}{k}}}-
\ln{\sin{\frac{\pi}{s}}}} {\frac{\varphi([k,s])}{2\rho(k,s)}\cdot
\left(
\ln{\frac{4}{\sqrt{2\gamma_0}}}-\frac{\ln{\gamma(k)}}{\varphi(k)}-
\frac{\ln{\gamma(s)}}{\varphi(s)}\right)}\right],
\label{Gamma3.6.4}
\end{equation}
and finally we obtain the bound \eqref{case2.12} which is
\begin{equation}
N=[\bk:\bq]\le \left[\frac {\ln{\frac{4\cdot
14^2}{\sqrt{\gamma_0}}}-\ln{\sin{\frac{\pi}{k}}}-
\ln{\sin{\frac{\pi}{s}}}} {\frac{\varphi([k,s])}{2\rho(k,s)}\cdot
\left(
\ln{\frac{4}{\sqrt{2\gamma_0}}}-\frac{\ln{\gamma(k)}}{\varphi(k)}-
\frac{\ln{\gamma(s)}}{\varphi(s)}\right)}\right]\cdot
\frac{\varphi([k,s])}{2\rho(k,s)}\ . \label{Gamma3.6.5}
\end{equation}
If either a pair $(k,s)$ is exceptional, or the right hand side of
\eqref{Gamma3.6.5} is more than $138$ (these are possible only for
pairs $(k,s)$ with $4\le s\le 11$ and $s\le k\le 870$), we also
apply to the pair $(k,s)$ the method A of the Case 2 to improve
the poor bound \eqref{Gamma3.6.4} of $N_0=[\bk:\bff_{k,s}]$ for
non-exceptional $(k,s)$. We can apply this method to any pair
$(k,s)$ since \eqref{case2.14} is valid for all $k\ge s\ge 3$ if
$a=2\gamma_0$. We obtain that $[\bk:\bq]\le 138$ for all $k\ge
s\ge 4$.

Let us assume that $s=3$ is exceptional. It means that either
$m_1=3$ or $m_3=3$ for $\Gamma^{(6)}_3(14)$. For example, let
$m_1=3$. Let us consider the V-arithmetic graph defined by $e$,
$\delta_1$ and $\delta_4$. We denote $\alpha=a_{14}^2$ where
$a_{14}=\delta_1\cdot \delta_4$. The determinant of the Gram matrix of
$e$, $\delta_1$ and $\delta_4$ is equal to $-6+2\alpha$. It
follows that $0<\sigma(\alpha)<3$ for $\sigma\not=\sigma^{(+)}$,
and $4<\sigma^{(+)}(\alpha)<14^2$. Then $\bk=\bq(\alpha)$. We
apply Theorems \ref{th121} and \ref{th122} to $\bff=\bq$ and
$\alpha$ where $M=1$, $B=2$, $R=\sqrt{3}/2$, $S=2\cdot e\cdot
14^2/3$. We obtain $[\bk:\bq]\le 76$.

Thus, finally, $[\bk:\bq]\le 138$ for all graphs
$\Gamma^{(6)}_3(14)$.

\subsubsection{V-arithmetic pentagon graphs $\Gamma^{(7)}_1(14)$.}
\label{subsubsec:Gamma7.1} In this case, considering $(-\alpha)$ from
\eqref{alpha},
we apply the methods A and B of Case 2 in Sec.
\ref{subsec:genresults} to $a=\gamma_0=(\sqrt{5}-1)^5\le
2.8854382$, $b_1=-14^5$, $b_2=-2^5$ (then $b=14^5$), and
$k:=\max\{m_1,m_2\}$, $s:=\min\{m_1,m_2\}$ where $k\ge s\ge 3$ and
$s_0=3$.

At first, we apply the Method B. Exceptional $l\ge 3$ satisfy
\eqref{case2.excepl} which is
\begin{equation}
\ln{\frac{4}{\sqrt{\gamma_0}}}-\frac{\ln{\gamma(l)}}{\varphi(l)}\le
0. \label{Gamma1.7.0}
\end{equation}
It follows that all $l\ge 3$ are non-exceptional.

All exceptional pairs $(k,s)$ where $k\ge s \ge 3$ that is when
\eqref{case2.2} which is
\begin{equation}
\ln{\frac{4}{\sqrt{\gamma_0}}}-\frac{\ln{\gamma(k)}}{\varphi(k)}-
\frac{\ln{\gamma(s)}}{\varphi(s)}\le 0
\label{Gamma1.7.1}
\end{equation}
satisfies, are $(k,s=3)$ where $k=3,4,5,7$.

We can take $K_0=324$ in \eqref{case2.6}. Then
$$
\Delta_1(2\gamma_0)=\ln{\frac{4}{\sqrt{\gamma_0}}}-
\frac{\ln{3}}{2}-\frac{\ln{331}}{330} \ge 0.28956765\ ,
$$
and $K_1=1262$ can be taken in \eqref{case2.10}. Checking
\eqref{case2.3} for $3\le s\le k<1262$, we obtain that $3\le s\le
90$ and $3\le s\le k\le 240$. Moreover, $k\le 126$ for $6\le
s\le 90$. For all these pairs $(k,s)$ satisfying \eqref{case2.3}
which is
\begin{equation}
\frac{\varphi([k,s])}{2\rho(k,s)}\cdot \left(
\ln{\frac{4}{\sqrt{\gamma_0}}}-\frac{\ln{\gamma(k)}}{\varphi(k)}-
\frac{\ln{\gamma(s)}}{\varphi(s)}\right) < 
\ln{\sqrt{\frac{14^5}{\gamma_0}}}-
\ln{\sin{\frac{\pi}{k}}}-
\ln{\sin{\frac{\pi}{s}}}\, ,
\label{Gamma1.7.2}
\end{equation}
we obtain
\begin{equation}
[\bff_{k,s}:\bq]=\frac{\varphi([k,s])}{2\rho(k,s)}\le 36
\label{Gamma1.7.3}
\end{equation}
where $36$ is achieved for $(k,s)=(73,3)$. Moreover, for all
these non-exceptional pairs $(k,s)$ we obtain the bound
\eqref{case2.11} which is
\begin{equation}
N_0=[\bk:\bff_{k,s}]\le \left[\frac {\ln{\sqrt{\frac{14^5}{\gamma_0}}}-
\ln{\sin{\frac{\pi}{k}}}-
\ln{\sin{\frac{\pi}{s}}}} {\frac{\varphi([k,s])}{2\rho(k,s)}\cdot
\left(
\ln{\frac{4}{\sqrt{\gamma_0}}}-\frac{\ln{\gamma(k)}}{\varphi(k)}-
\frac{\ln{\gamma(s)}}{\varphi(s)}\right)}\right],
\label{Gamma1.7.4}
\end{equation}
and finally we obtain the bound \eqref{case2.12} which is
\begin{equation}
N=[\bk:\bq]\le \left[\frac {\ln{\sqrt{\frac{14^5}{\gamma_0}}}-
\ln{\sin{\frac{\pi}{k}}}-\ln{\sin{\frac{\pi}{s}}}}
{\frac{\varphi([k,s])}{2\rho(k,s)}\cdot
\left(
\ln{\frac{4}{\sqrt{\gamma_0}}}-\frac{\ln{\gamma(k)}}{\varphi(k)}-
\frac{\ln{\gamma(s)}}{\varphi(s)}\right)}\right]\cdot
\frac{\varphi([k,s])}{2\rho(k,s)}\ . \label{Gamma1.7.5}
\end{equation}
If either a pair $(k,s)$ is exceptional, or the right hand side of
\eqref{Gamma1.7.5} is more than $36$ (these are possible only for
pairs $(k,s)$ with $3\le s\le 5$ and $s\le k\le 240$), we also
apply to the pair $(k,s)$ the method A of the Case 2 to improve
the poor bound \eqref{Gamma1.7.4} of $N_0=[\bk:\bff_{k,s}]$ for
non-exceptional $(k,s)$. We can apply this method to any pair
$(k,s)$ since \eqref{case2.14} is valid for all $k\ge s\ge 3$ if
$a=\gamma_0$. 

For all $3\le s\le 5$ and $s\le k\le 240$, the Method A gives 
$[\bk:\bq]\le 36$ except $s=k=3$ (equivalently, $m_1=m_2=3$). Finally 
$[\bk:\bq]\le 42$ for all graphs $\Gamma^{(7)}_1(14)$.

\subsubsection{V-arithmetic pentagon graphs $\Gamma^{(7)}_2(14)$.}
\label{subsubsec:Gamma7.2} Since $0<\sigma(\sin^2{(\pi/m_2)})<1$
for any embedding $\sigma:\bk\to \br$, in this case,
we have for $\alpha$ from \eqref{alpha}
the same inequalities \eqref{inegamma36} and \eqref{inegamma+3627}
as for $\Gamma^{(6)}_3(14)$. Thus,
we get the same upper bound $[\bk:\bq]\le 138$ as
for $\Gamma_3^{(6)}(14)$.

\medskip

This finishes the proof of Theorem \ref{thpentgraphs}.

\medskip

%%%%%%%%%%%%%%%%%%%%%%%%%%%%
%%%%%%%%%%%%%%%%%%%%%%%%%%%%%
%%%%%%%%%%%%%%%%%%%%%%%%%%%%
%%%%%%%%%%%%%%%%%%%%%%%%%%%%
%%%%%%%%%%%%%%%%%%%%%%%%%%%%%

\section{Appendix: Some results about\\ cyclotomic fields} \label{appendix1}

Here we give some results about cyclotomic fields which we used. All of 
them follow from standard results. For example, see the book \cite{CF}.

We consider the cyclotomic field $\bq\left(\sqrt[l]{1}\right)$ and
its totally real subfield $\bff_l=\bq\left(\cos{(2\pi/l)}\right)$.
We have $[\bq\left(\sqrt[l]{1}\right):\bq]=\varphi(l)$ where 
$\varphi(l)$ is the Euler function. We have
$\bff_l=\bq\left(\sqrt[l]{1}\right)=\bq$ for $l=1,2$, and
$[\bff_l:\bq]=\varphi(l)/2$ for $l\ge 3$. It is known (e.g., see
\cite{CF}) that the discriminant of the field
$\bq\left(\sqrt[l]{1}\right)$ is equal to (where $p$ is prime)
\begin{equation}
|{\rm discr}\,\bq(\sqrt[l]{1})|=\frac{l^{\varphi(l)}}{\prod_{p|l}
{p^{\varphi(l)/(p-1)}}}\ . \label{discrql}
\end{equation}

Let $\zeta_l=\exp{(2\pi i/l)}$ be a primitive $l$-th root of $1$.
The element $\zeta_l$ generates the ring of integers of
$\bq\left(\sqrt[l]{1}\right)$. Further we assume that $l\ge 3$.
The equation of $\zeta_l$ over $\bff_l$ is
$g(x)=(x-\zeta_l)(x-\zeta_l^{-1})=x^2-(\zeta_l+\zeta_l^{-1})x+1=0$.
We have
$g^\prime(\zeta_l)=2\zeta_l-(\zeta_l+\zeta_l^{-1})=\zeta_l-\zeta_l^{-1}$.
Thus,
$$
N_{\bq\left(\sqrt[l]{1}\right)/\bff_l}(g^\prime(\zeta_l))=
(\zeta_l-\zeta_l^{-1})(\zeta_l^{-1}-\zeta_l)=4\sin^2{(2\pi/l)}
$$
which gives the discriminant $\delta
\left(\bq\left(\sqrt[l]{1}\right)/\bff_l\right)=4\sin^2{(2\pi/l)}$.
It follows
$$
|\delta\left(\bq\left(\sqrt[l]{1}\right)/\bq\right)|=
|\delta\left(\bff_l/\bq\right)^2N_{\bff_l/\bq}(4\sin^2{(2\pi/l)})|.
$$
We have
\begin{equation}
N_{\bff_l/\bq }(4\sin^2{(\pi/l)})=\gamma(l)=\left\{
\begin{array}{cl}
p &\ {\rm if}\  l=p^t>2\ {\rm where}\ p\ {\rm is}\ {\rm prime ,} \\
1 &\ {\rm otherwise.}
\end{array}\
\label{normsin}
 \right .
\end{equation}

If $l$ is odd, then $4\sin^2{(\pi/l)}$ and $4\sin^2{(2\pi/l)}$ are
conjugate, and their norms are equal. Thus,
$$
N_{\bff_l/\bq}(4\sin^2{(2\pi/l)})=\gamma(l),\ \text{if $l\ge 3$ is
odd.}
$$

If $l$ is even and $l_1=l/2$, then
$4\sin^2{(2\pi/l)}=4\sin^2{(\pi/l_1)}$. If $l_1$ is odd, then
$\bff_{l_1}=\bff_{l}$, and we get
$$
N_{\bff_l/\bq}(4\sin^2{(2\pi/l)})=\gamma(l/2)\ \text{if $l\ge 6$
is even, but $l/2$ is odd.}
$$
If $l_1\ge 4$ is even, then $[\bff_l:\bff_{l_1}]=2$, and we get
$$
N_{\bff_l/\bq}(4\sin^2{(2\pi/l)})=\gamma(l/2)^2\ \text{if $l\ge 8$
and $l/2$ is even.}
$$
At last,
$$
N_{\bff_4/\bq}(4\sin^2{(2\pi/4)})=4
$$
if $l=4$.

Thus, finally we get for $l\ge 3$:
\begin{equation}
N_{\bff_l/\bq }(4\sin^2{(2\pi/l)})=\widetilde{\gamma}(l)=\left\{
\begin{array}{cl}
\gamma(l) &\ {\rm if}\ l\ge 3\ {\rm is\ odd,} \\
\gamma(l/2) &\ {\rm if}\ l/2 \ge 3\ {\rm is\ odd,}\\
\gamma(l/2)^2 &\ {\rm if}\ l/2\ge 4\ {\rm is\ even,}\\
4 &\ {\rm if}\ l=4. \\
\end{array}\
\label{normsin2}
 \right .
\end{equation}
Moreover, we obtain the formula for the discriminant:
\begin{equation}
|{\rm discr}\,\bff_l|= \left(|{\rm discr}\,\bq
(\sqrt[l]{1})|\left/\right.\widetilde{\gamma}(l)\right)^{1/2}\
\text{for}\ l\ge 3
\label{discrFl}
\end{equation}
where $|{\rm discr}\,\bq (\sqrt[l]{1})|$ is given by
\eqref{discrql}, and $\widetilde{\gamma}(l)$ is given by
\eqref{normsin2}.

We denote
$\bff_{k,s}=\bq\left(\cos{(2\pi/k)},\,\cos{(2\pi/s)}\right)$.
Further we assume that $k,s\ge 3$. Let $m=[k,s]$ be the least
common multiple of $k$ and $s$. Then $\bff_{k,s}\subset
\bff_m\subset \bq(\sqrt[m]{1})$. We have
$\text{Gal}\,\left(\bq(\sqrt[m]{1})/\bq\right)=(\bz/m\bz)^\ast$
where $\alpha\in (\bz/m\bz)^\ast$ acts on each $m$-th root $\zeta$
of $1$ by the formula $\zeta\mapsto \zeta^\alpha$. Obviously,
$\bff_{k,s}$ is the fixed field of the subgroup $G$ of the Galois
group $(\bz/m\bz)^\ast$ which consists of all $\alpha \in
(\bz/m\bz)$ such that $\alpha\equiv \pm 1\mod k$ and $\alpha\equiv
\pm 1\mod s$. The $G$ includes the subgroup of order two of
$\alpha\equiv \pm 1\mod m$. If $\alpha\equiv 1\mod k$, then
$\alpha\equiv 1+kt\mod m$ where $t\in \bz$. If $1+kt\equiv -1\mod
s$, then the equation $kt+sr=2$ has an integer solution $(t,r)$
which is equivalent to $(k,s)|2$. Thus, $G$ has the order $4$ if
and only if  $(k,s)|2$. Otherwise, $G$ has the order $2$. We set
\begin{equation}
\rho(k,s)=\left\{
\begin{array}{cl}
2 &\ {\rm if}\ (k,s)|2, \\
1 &  {\rm otherwise,}\\
\end{array}\
\label{defrho}
 \right.
\end{equation}
and we obtain
\begin{equation}
[\bff_{k,s}:\bq]=\frac{\varphi(m)}{2\rho(k,s)}. \label{degreekl}
\end{equation}
Moreover, we get
$$
\bff_{k,s}=\bff_m\ \text{if}\ (k,s)\not|\, 2.
$$
It follows,
\begin{equation}
|\text{discr}\,\bff_{k,s}|=|\text{discr}\,\bff_m|\ \text{if}\
(k,s)\not|\,2, \label{discrFks1}
\end{equation}
where $m=[k,s]$, and $|\text{discr}\,\bff_m|$ is given by
\eqref{discrFl}.

Assume that $(k,s)|2$. If $(k,s)=1$, then the fields
$\bq(\sqrt[k]{1})$ and $\bq(\sqrt[s]{1})$ are linearly disjoint
and their discriminants are coprime. Then their subfields $\bff_k$
and $\bff_s$ are linearly disjoint, and their discriminants are
coprime, and we obtain
$$
|\text{discr}\,\bff_{k,s}|=|\text{discr}\,\bff_k|^{(\varphi(s)/2)}
|\text{discr}\,\bff_s|^{(\varphi(k)/2)}\ \text{if}\ (k,s)=1\
\text{and}\ k,s\ge 3.
$$

Assume that $(k,s)=2$. Then one of $k/2$ or $s/2$ is odd. Assume,
$k_1=k/2$ is odd. Then $\bff_k=\bff_{k_1}$ and
$\bff_{k,s}=\bff_{k_1,s}$ where $(k_1,s)=1$. Thus, we obtain the
previous case which gives exactly the same formula. We finally
obtain
\begin{equation}
|\text{discr}\,\bff_{k,s}|=|\text{discr}\,\bff_k|^{(\varphi(s)/2)}
|\text{discr}\,\bff_s|^{(\varphi(k)/2)}\ \text{if}\ (k,s)|2\
\text{and}\ k,s\ge 3 \label{discrFks2}
\end{equation}
where the discriminants $|\text{discr}\,\bff_k|$ and
$|\text{discr}\,\bff_s|$ are given by \eqref{discrFl}.

%%%%%%%%%%%%%%%%%%%%%%%%%%%%%%%%%%%%%%%
%%%%%%%%%%%%%%%%%%%%%%%%%%%%%%%%%%%%%%%%
%%%%%%%%%%%%%%%%%%%%%%%%%%%%%%%%%%%%%%%%
%%%%%%%%%%%%%%%%%%%%%%%%%%%%%%%%%%%%%%%%%

V.V. Nikulin \par Deptm. of Pure Mathem. The University of
Liverpool, Liverpool\par L69 3BX, UK; \vskip1pt Steklov
Mathematical Institute,\par ul. Gubkina 8, Moscow 117966, GSP-1,
Russia

vnikulin@liv.ac.uk \ \ vvnikulin@list.ru
%%%%%%%%%%%%%%%%


\newpage

\begin{thebibliography}{ADSE}

\bibitem{Agol} I. Agol, \emph{Finiteness of arithmetic Kleinian
reflection groups}, Proceedings of the International Congress of
Mathematicians, Madrid, 2006, Vol. 2, 951--960
(see also math.GT/0512560).

\bibitem{ABSW} I. Agol, M. Belolipetsky, P. Storm, K. Whyte,
\emph{Finiteness of arithmetic hyperbolic reflection groups}, Preprint 2006,
Arxiv:math.GT/0612132, 16 pages.

\bibitem{Bor} R.E. Borcherds, \emph{Monstrous moonshine and monstrous
Lie superalgebras},  Invent. Math.  \textbf{109}  (1992),  no. 2, 405--444.

\bibitem{Borel} A. Borel, \emph{Commensurability classes and volumes
of hyperbolic 3-manifolds}, Ann. Scuola Norm. Sup. Pisa, \textbf{8}
(1981), 1-- 33.

\bibitem{CF} J.W.S. Cassels, A. Frohlich (eds), \emph{Algebraic number
theory}, Proc. of an instructional conference organized by LMS,
Academic Press, 1967, xviii+366 pp.

\bibitem{CN} J.H. Conway, J. H., S.P. Norton, \emph{Monstrous moonshine},
Bull. London Math. Soc.  \textbf{11}  (1979), no. 3, 308--339.

\bibitem{Coxeter} H.S.M. Coxeter, \emph{Discrete groups generated by
reflections}, Ann. of Math. \textbf{35} (1934), no. 2, 588--621.

\bibitem{GrNik1} V.A. Gritsenko, V.V. Nikulin,
\emph{Automorphic forms and Lorentzian Kac-Moody algebras. I},
Internat. J. Math. \textbf{9} (1998), no. 2, 153--199.

\bibitem{GrNik2} V.A. Gritsenko, V.V. Nikulin,
\emph{Automorphic forms and Lorentzian Kac-Moody algebras. II},
 Internat. J. Math.  \textbf{9}  (1998),  no. 2, 201--275.

\bibitem{GrNik3} V.A. Gritsenko, V.V. Nikulin,
\emph{The arithmetic mirror symmetry and Calabi-Yau manifolds},
Comm. Math. Phys.  \textbf{210}  (2000),  no. 1, 1--11.

\bibitem{GrNik4} V.A. Gritsenko, V.V. Nikulin,
\emph{On the classification of Lorentzian Kac-Moody algebras},
Uspekhi Mat. Nauk  \textbf{57}  (2002)  no. 5, 79--138; English transl.
in  Russian Math. Surveys  \textbf{57}  (2002),  no. 5,
921--979.

\bibitem{Lan} F. Lann\'er
\emph{On complexes with transitive groups of automorphisms},
Comm. S\'em., Math. Univ. Lund \textbf{11}  (1950), 1--71.

\bibitem{LMR} D.D. Long, C. Maclachlan, A.W. Reid,
\emph{Arithmetic Fuchsian groups of genus zero}, Pure and Applied
Mathematics Quarterly, \textbf{2} (2006), no. 2, 1--31.

\bibitem{Nik1} V.V. Nikulin,
\emph{On arithmetic groups generated by reflections in Lobachevsky spaces},
Izv. Akad. Nauk SSSR Ser. Mat. \textbf{44} (1980), no. 3, 637--669;
English transl. in  Math. USSR Izv. \textbf{16} (1981), no. 3,
573--601.

\bibitem{Nik2} V.V. Nikulin, \emph{On the classification of arithmetic groups
generated by reflections in Lobachevsky spaces}, Izv. Akad. Nauk SSSR Ser.
Mat. \textbf{45} (1981), no. 1, 113--142;  English transl. in
Math. USSR Izv. \textbf{18} (1982), no. 1, 99--123.

\bibitem{Nik3} V.V. Nikulin,
\emph{Discrete reflection groups in Lobachevsky spaces
and algebraic surfaces}, Proceedings of the International Congress
of Mathematicians, Berkeley, 1986, Vol. 1, 1987,  pp. 654-671.

\bibitem{Nik4} V.V. Nikulin
\emph{The remark on arithmetic groups generated by reflections in Lobachevsky
spaces}, Preprint Max-Planck-Institut f\"ur Mathematik, Bonn. MPI/91 (1991).

\bibitem{Nik5} V.V. Nikulin,
\emph{Finiteness of the number of arithmetic groups generated by reflections
in Lobachevski spaces}, Izvestiya:Mathematics, \textbf{71} (2007), no. 1,
53--56 (see also arXiv.org:math.AG/0609256).


\bibitem{Nik6} V.V. Nikulin,
\emph{On ground fields of arithmetic hyperbolic reflection groups}, Preprint
arXiv:0708.3991 [math.AG], 33 pages.


\bibitem{Tak1} K. Takeuchi,
\emph{A characterization of arithmetic Fuchsian groups,}
J. Math. Soc. Japan, \textbf{27} (1975), no. 4, 600--612.

\bibitem{Tak2} K. Takeuchi,
\emph{Arithmetic triangle groups,} J. Math. Soc. Japan \textbf{29} (1977),
no. 1, 91--106.

\bibitem{Tak3} K. Takeuchi,
\emph{Commensurability classes of arithmetic triangle groups,}
J. Fac. Sci. Univ. Tokyo Sect. 1A Math. \textbf{24} (1977), no. 1, 201-212.

\bibitem{Tak4} K. Tekeuchi,
\emph{Arithmetic Fuchsian groups with signature $(1,e)$,}
J. Math. Soc. Japan \textbf{35} (1983), no. 3, 381--407.

\bibitem{Vig} M.-F. Vigneras,
\emph{Quelques remarques sur la conjecture $\lambda_1\ge \frac{1}{4}$},
Seminar on number theory, Paris 1981-82, Progr. Math., \textbf{38},
Birkhauser,
(1983), 321--343.

\bibitem{Vin1} \'E.B. Vinberg, \emph{Discrete groups generated by
reflections in Loba\v cevski\v i spaces}, Mat. Sb. (N.S.) \textbf{72}
(1967), 471--488; English transl. in Math. USSR Sb.\textbf{1} (1967),
429--444.

\bibitem{Vin2} \'E.B. Vinberg, \emph{The nonexistence of crystallographic
reflection groups in Lobachevski\v i spaces of large dimension},
Funkts. Anal. i Prilozhen. \textbf{15} (1981), no. 2, 67--68; English transl.
in Funct. Anal. Appl.  \textbf{15} (1981), 216--217.

\bibitem{Vin3} \'E.B. Vinberg, \emph{Absence of crystallographic
reflection groups in Lobachevski\v i spaces of large dimension},
Trudy Moskov. Mat. Obshch. \textbf{47} (1984), 68--102; English transl.
in Trans. Moscow Math. Soc. \textbf{47} (1985).

\bibitem{Vin4} \'E.B. Vinberg, \emph{Discrete reflection groups in
Lobachevsky spaces},  Proceedings of the International Congress of
Mathematicians, Warsaw, 1983, Vol. 1, 1984, pp. 593--601.

\bibitem{Vin5} \'E.B. Vinberg
\emph{Hyperbolic groups of reflections},
Uspekhi Mat. Nauk \textbf{40} (1985), no. 1, 29--66;
English transl. in Russian. Math. Surv. \textbf{40} (1985), no. 1,
31--75.

\bibitem{Vinog} I.M. Vinogradov, \emph{An introduction to the theory of
numbers}, Pergamon, London, 1955.

\bibitem{Zog} P. Zograf, \emph{A spectral proof of Rademacher's
conjecture for congruence subgroups of the modular group}, J. f\"ur
die reine agew. Math. \textbf{414} (1991), 113--116.






\end{thebibliography}
\end{document}